\documentclass[11pt,letterpaper]{article}
\usepackage[T1]{fontenc}
\usepackage[utf8]{inputenc}
\usepackage[margin=1in]{geometry}
\usepackage{lmodern}
\usepackage{microtype}

\usepackage[USenglish]{babel}
\usepackage{graphicx}
\usepackage{multirow}
\usepackage{amsmath,amssymb,amsfonts,bbm}
\usepackage{amsthm}
\usepackage{mathrsfs}
\usepackage{xcolor}
\usepackage{textcomp}
\usepackage{manyfoot}
\usepackage{booktabs}
\usepackage{xpatch}
\usepackage{enumitem}
\usepackage{thmtools}
\usepackage{thm-restate}
\usepackage{etoolbox}
\usepackage{dsfont}
\usepackage{mathtools}

\usepackage{tabularx}
\usepackage{scalerel}
\usepackage[numbers,sort&compress]{natbib}
\usepackage[hypertexnames=false]{hyperref}
\usepackage[capitalise]{cleveref}
\hypersetup{hidelinks}
\allowdisplaybreaks[1]
\newcommand{\PaperTitle}{Repair-Based Duality for Distributionally Robust Optimization via Conditional-Moment Optimal Transport}

\newcommand{\PaperKeywords}{distributionally robust optimization, duality, perturbation analysis}

\newcommand{\PaperAbstract}{%
Distributionally robust optimization (DRO) models often admit dual representations that replace optimization over probability distributions with more tractable formulations, but establishing primal-dual equality can be challenging when ambiguity sets involve conditional moment restrictions, noncompact uncertainty spaces, or unbounded losses. We study these questions through the conditional-moment optimal transport (CM-OT) framework, which encompasses transport-, divergence-, and jointly perturbed DRO models. Our main contribution is a repair-based sufficient condition for primal-dual equality. The approach constructs feasible couplings from nearly feasible ones by mixing with suitable reference couplings, while a local conditional uniform integrability condition on positive-loss tails ensures that the resulting repairs do not incur a nonvanishing loss in objective value. This establishes the closedness of the associated perturbation value function without relying on compactness or conventional interiority conditions. We further distinguish primal-dual equality from the additional steps required to obtain pointwise dual representations, providing conditions under which essentially bounded measurable conditional-moment multipliers can be replaced, at the level of optimal values, by bounded continuous multipliers and under which optimization and integration can be interchanged. Finally, we apply the framework to a class of DRO models combining divergence-based ambiguity with distributional calibration, including contextual resource-allocation and reliability applications.
}

\newcommand{\PaperAppendix}{\appendix}
\newcommand{\PaperEndAppendix}{}
\newcommand{\PaperBibStyle}{plainnat}

\hypersetup{%
  pdftitle={\PaperTitle},%
  pdfauthor={Louis L. Chen and Jake Roth},%
  pdfkeywords={\PaperKeywords}%
}

\makeatletter
\newcommand{\setitemlabel}[1]{\def\@currentlabel{#1}}
\makeatother

\makeatletter
\newcommand{\textlabel}[2]{%
  \phantomsection%
  \def\@currentlabel{#2}%
  \label{#1}%
  #2%
}
\makeatother

\newcommand{\PaperSupplementSection}[1]{Appendix~\ref{#1}}

\newcommand{\PaperFootnote}[1]{\footnote{#1}}



\newcommand{\mathopfont}{\mathsf}

\providecommand{\argmin}{\operatorname*{arg\,min}}

\providecommand{\argmax}{\operatorname*{arg\,max}}

\providecommand{\projection}{\operatorname*{\mathopfont{proj}}}

\newcommand{\reals}{\mathbbm{R}}
\newcommand{\Reals}{\overline{\mathbbm{R}}}

\newcommand{\R}{\mathbbm{R}}
\newcommand{\ind}{\mathds{1}}
\DeclarePairedDelimiter\abs{\lvert}{\rvert}%

\DeclarePairedDelimiterXPP\sign[1]{\operatorname{\mathopfont{sign}}}[]{}{#1}%
\providecommand{\intr}{\operatorname{\mathopfont{int}}}
\providecommand{\ri}{\operatorname{\mathopfont{ri}}}

\providecommand{\cl}{\operatorname{\mathopfont{cl}}}
\providecommand{\aff}{\operatorname{\mathopfont{aff}}}
\providecommand{\dom}{\operatorname{\mathopfont{dom}}}
\providecommand{\epi}{\operatorname{\mathopfont{epi}}}
\providecommand{\hypo}{\operatorname{\mathopfont{hypo}}}

\newcommand{\medint}{%
  \mathop{\scaleobj{0.75}{\int}\!}\nolimits
}%

\renewcommand{\complement}{\mathopfont{c}}

\DeclarePairedDelimiter\norm{\lVert}{\rVert}%
\DeclarePairedDelimiterX{\inner}[2]{\langle}{\rangle}{#1, #2}
\newcommand{\st}{\textnormal{s.t.}}

\DeclarePairedDelimiterXPP\mydet[1]{\operatorname{\mathopfont{det}}}[]{}{#1}%

\DeclarePairedDelimiterXPP\myvec[1]{\operatorname{\mathopfont{vec}}}[]{}{#1}%

\newtoggle{detail}
\toggletrue{detail}    %
\newtoggle{newargument}
\toggletrue{newargument}    %

\theoremstyle{plain}%
\newtheorem{theorem}{Theorem}
\newtheorem{proposition}{Proposition}
\newtheorem{corollary}{Corollary}
\newtheorem{lemma}{Lemma}
\newtheorem{assumption}{Assumption}

\theoremstyle{definition}%
\newtheorem{definition}{Definition}

\newtheorem{philemma}{Lemma}

\crefname{philemma}{Lemma}{Lemmas}
\Crefname{philemma}{Lemma}{Lemmas}
\newtheorem{phiassumption}{Assumption}

\crefname{phiassumption}{Assumption}{Assumptions}
\Crefname{phiassumption}{Assumption}{Assumptions}

\theoremstyle{remark}%
\newtheorem{example}{Example}
\newtheorem{remark}{Remark}
\crefname{example}{Example}{Examples}
\crefname{assumption}{Assumption}{Assumptions}

\newcommand{\exampleendsymbol}{\ensuremath{\diamond}}
\AtEndEnvironment{example}{%
\unskip\nobreak\hfill\exampleendsymbol\par
}

\newcounter{subexample}[example]          %

\crefname{subexample}{Example}{Examples}

\newcounter{inlineenum}

\newcommand{\U}{\mathcal{U}}
\newcommand{\V}{\mathcal{V}}
\newcommand{\W}{\mathcal{W}}
\newcommand{\X}{\mathcal{X}}

\newcommand{\Z}{\mathcal{Z}}

\newcommand{\primal}{\textnormal{P}}
\newcommand{\Pdot}{$\breve{\textnormal{P}}$}
\newcommand{\pdot}{\breve{p}}
\newcommand{\qdot}{\breve{q}}
\newcommand{\Fdot}{\breve{F}}
\newcommand{\primalExplicit}{\Pdot}
\newcommand{\pcw}{\textnormal{D}}
\newcommand{\ip}{\textnormal{IP}}

\newcommand{\dualLinftySym}{\pcw_{L^\infty(\hat\zeta)}}
\newcommand{\dualLoneSym}{\pcw_{L^1(\hat\zeta)}}
\newcommand{\dualCbSym}{\pcw_{\mathcal{C}_b(\Z)}}
\makeatletter
\newcommand{\dualPsiAnchor}{\Hy@raisedlink{\hyper@anchorstart{dualPsi}\hyper@anchorend}}
\newcommand{\dualPsiIPAnchor}{\Hy@raisedlink{\hyper@anchorstart{dualPsiIP}\hyper@anchorend}}
\DeclareRobustCommand{\dualPsiRef}{\hyperlink{dualPsi}{\textup{\tagform@{\ref*{eq:dual_Psi}}}}}
\DeclareRobustCommand{\dualPsiIPRef}{\hyperlink{dualPsiIP}{\textup{\tagform@{\ref*{eq:dual_Psi_IP}}}}}
\makeatother
\newcommand{\dualLinfty}{\hyperlink{dualPsi}{\dualLinftySym}}
\newcommand{\dualLone}{\hyperlink{dualPsi}{\dualLoneSym}}
\newcommand{\dualCb}{\hyperlink{dualPsi}{\dualCbSym}}
\newcommand{\dualLinftyIP}{\hyperlink{dualPsiIP}{\dualLinftySym+\ip}}

\newcommand{\dualCbIP}{\hyperlink{dualPsiIP}{\dualCbSym+\ip}}

\begin{document}

\title{\PaperTitle}
\author{%
  Louis L. Chen\thanks{These authors contributed equally to this work.}\\[0.3em]
  {\small Department of Operations Research}\\
  {\small Naval Postgraduate School}\\
  {\small 1 University Circle}\\
  {\small Monterey, California 93943, USA}\\
  {\small \href{mailto:louis.chen@nps.edu}{\texttt{louis.chen@nps.edu}}}
  \and
  Jake Roth\footnotemark[1]\\[0.3em]
  {\small School of Statistics}\\
  {\small University of Minnesota}\\
  {\small 313 Church Street SE}\\
  {\small Minneapolis, Minnesota 55455, USA}\\
  {\small \href{mailto:rothjakem@gmail.com}{\texttt{rothjakem@gmail.com}}}
}
\date{}
\maketitle

\begin{abstract}
\PaperAbstract
\end{abstract}

\section{Introduction}
\label{sec:introduction}

Distributionally robust optimization (DRO) evaluates decisions against a family of plausible probability distributions rather than a single prescribed distribution. Different specifications of this family capture different forms of distributional ambiguity, including uncertainty in moments and deviations measured by statistical divergences or transportation costs. Despite these modeling differences, a common analytical challenge is to evaluate the worst-case expected loss over the ambiguity set. Duality is central to this task, replacing optimization over probability distributions with alternative representations that support theoretical analysis and, for suitably structured models, tractable reformulations. Establishing primal-dual equality is therefore an essential step in making a DRO formulation useful.

A setting in which these duality questions can be studied collectively is the conditional-moment optimal transport (CM-OT) framework introduced by \citet{jiajin2025unifying}. This framework augments optimal transport with auxiliary weights and conditional moment restrictions, allowing ambiguity in outcomes and likelihood ratios to be represented within a common coupling formulation. It encompasses Wasserstein, generalized $\phi$-divergence, and Sinkhorn ambiguity models, as well as models that jointly perturb outcomes and likelihood ratios. Conditional moments thus serve not merely as additional restrictions on a particular DRO problem, but as a device for representing several familiar forms of distributional ambiguity within a unified optimization structure.

The breadth of this representation introduces difficulties that do not arise in the same form in conventional transport DRO. When the conditioning variable has infinite support, a conditional moment equality generally imposes infinitely many constraints, and strict feasibility of a scalar transport budget may not provide a useful interiority condition for the resulting feasible region. Noncompact uncertainty spaces further limit the applicability of compactness-based arguments.
Unbounded losses introduce a separate issue in that small constraint violations can be accompanied by large objective contributions concentrated on events of small probability.
These considerations motivate a duality analysis that addresses
CM-OT feasibility and objective tails together, rather than
treating feasibility and objective behavior as separate issues.

The convex perturbation duality framework provides a transparent way to do so, and in particular, distinguishes the regularity needed
for primal-dual equality from the particular assumptions used to establish it.
In this framework, the
relevant property is closedness of the optimal-value function at the unperturbed problem. For the worst-case maximization problems considered here, this means that vanishing constraint perturbations cannot produce
a persistent improvement over the original optimal value. Classical routes to this property include compactness arguments, together with
appropriate semicontinuity, as well as Slater-type constraint qualifications, which establish continuity through suitable interiority in the perturbation space. Such interiority can be delicate to verify in
optimization over measures, particularly when conditional moment restrictions give rise to function-valued perturbations. In that setting, a neighborhood-interiority condition requires feasibility
throughout a neighborhood of conditional-moment targets in the chosen perturbation space, which does not follow merely from strict feasibility of the scalar transport budget. The essential requirement for primal-dual equality, however, concerns the limiting optimal values, rather than either this neighborhood feasibility or convergence of the couplings themselves.

In this work, we establish the requisite closedness for primal-dual equality in CM-OT through a \textit{repair mechanism}. The mechanism transforms an infeasible coupling to a feasible coupling, with any deterioration in objective value vanishing as the constraint violations vanish. 
For the mechanism to work as intended, we present \emph{repair-based} sufficient conditions involving the existence of
suitable reference couplings and local conditional uniform
integrability of positive-loss tails. 
The reference couplings allow
conditional moment and transport-budget violations to be corrected
through mixtures, while the tail condition prevents these corrections
from discarding a nonvanishing amount of objective value as
infeasibility tends to zero.
Together, these conditions establish
closedness of the perturbation value function and hence primal-dual
equality, allowing for noncompact uncertainty spaces and unbounded losses.

We complement this repair-based analysis by examining the additional
steps needed to obtain more refined representations of the dual problem. 
Indeed, the conditional
moment restriction gives rise to a function-valued dual multiplier,
making its admissible function space an additional consideration.
We give conditions under which essentially bounded measurable
dual multipliers can be replaced by bounded continuous dual multipliers without
changing the dual optimal value. Further, 
we also identify conditions under
which optimization over couplings in the dual can be interchanged
with integration, yielding a representation as an expectation of
pointwise suprema. 
Our analysis thus distinguishes the assumptions
that establish primal-dual equality from those that justify the
choice of dual multiplier feasible space and dual objective form. 

\subsection{Literature Review}
Duality results for distributionally robust optimization have developed alongside different representations of distributional ambiguity. For moment-based models, \citet{delage2010distributionally} study uncertainty in the mean and covariance, while \citet{wiesemann2014distributionally} develop a broader framework based on conic-representable distributional information and establish conditions for computational tractability. For divergence-based models, the connection between stochastic optimization, convex conjugacy, and $\phi$-divergence functionals appears in the work of \citet{bental1987penalty}. \citet{bental2011robust} subsequently derive tractable robust counterparts for optimization problems with $\phi$-divergence uncertainty in probability vectors, and \citet{bayraksan_love_2015} provide a systematic treatment of the resulting stochastic programming formulations and their duals. These works establish important links between the specification of distributional ambiguity and the structure of the associated optimization problem.

Transport-based ambiguity provides another setting in which duality has enabled both general analysis and tractable reformulations. \citet{esfahani_kuhn_2018} derive finite-dimensional convex reformulations for Wasserstein DRO centered at empirical distributions under suitable structural assumptions. \citet{murthy2019quantifying} establish strong duality for a broad class of transport costs on Polish spaces, while \citet{gao2023wasserstein} develop a constructive duality analysis that also characterizes worst-case distributions and the role of loss growth. These results already accommodate settings beyond compactly supported uncertainty and demonstrate that compactness is not a prerequisite for conventional transport DRO duality. Entropically regularized transport gives rise to another ambiguity specification, studied by \citet{wang2025sinkhorn} in Sinkhorn DRO.

Conditional restrictions on transport plans have also been studied as a source of distributional robustness. \citet{li2022martingale} investigate transport-based DRO under martingale constraints, establish a connection with Tikhonov regularization, and derive a corresponding strong-duality result. More directly related to the present work, \citet{jiajin2025unifying} introduce the CM-OT framework, which permits simultaneous perturbations of outcomes and likelihood ratios and represents several divergence- and transport-based models within a common formulation. Their general CM-OT duality theorem uses compactness together with semicontinuity assumptions. They also provide a specialized noncompact result involving a discrete reference distribution and an unconditional moment restriction. For the genuinely conditional setting, they emphasize the difficulty created by the continuum of moment restrictions and conjecture that their general theorem does not survive removal of compactness. Our work addresses the complementary question of which alternative regularity conditions restore duality when compactness is unavailable.

There are complementary perspectives to our work that also study duality through the regularity
of perturbed optimal values. \citet{shapiro2001duality} presents a
conjugate-duality treatment of conic linear problems, including
moment problems, that relates primal-dual equality and dual attainment
to properties of an associated optimal-value function. For DRO
with matrix moment constraints, \citet{xu2018distributionally}
examine Slater-type conditions and develop alternative sufficient
conditions based on closedness of the perturbed value function.
Their latter analysis uses tightness and uniform integrability of
the moment functions to establish weak compactness and stability
of the perturbed ambiguity sets. Our work shares this value-function
perspective but develops sufficient conditions tailored to the
conditional moment restrictions in CM-OT. 

A separate analytical issue is the passage from an abstract dual
problem to a representation involving an expectation of pointwise
suprema. In transport DRO, \citet{gao2022short} establish a general
duality result using one-dimensional convex analysis, conditional
on an interchangeability principle. They also characterize that principle through measurable projection and selection properties. Their work distinguishes the convex-analytic equality of optimal values from the interchange of optimization and integration required to express the dual as an expectation of pointwise suprema. This distinction is closely related to the classical theory of convex integral functionals developed by \citet{rockafellar1971integral}. We retain this separation in CM-OT, where the conditional moment restriction introduces a function-valued multiplier and makes the choice of multiplier space an additional part of the analysis.

A related literature studies duality for the transportation problem between two fixed marginals, where the value of such problems is referred to as a discrepancy (distance).
Discrepancies commonly give rise to ambiguity sets comprised of marginals with sufficiently small discrepancy from a reference marginal.
\citet{gozlan2015kantorovichdualitygeneraltransport} develop Kantorovich-type duality for general transport costs, including weak transport costs that depend on conditional distributions. Within this broader setting, \citet{carlier2026weak} study weak optimal transport with hard and soft moment constraints, including entropically regularized formulations. The analysis assumes compact underlying spaces and establishes primal-dual equality under their standing cost and regularity hypotheses. An additional moment qualification, requiring local realizability of compatible bounded residuals, is used to obtain the existence of optimal dual multipliers. 
However, our setting and analysis differ in that we are concerned with duality of the problem that optimizes over the ambiguity sets constructed from such transportation problems. 
We emphasize that the reference marginal is prescribed but the alternative marginal is optimized subject to a transport budget, and that this freedom features in our analysis, whereby state-dependent mixing of conditional distributions can preserve the reference marginal while correcting conditional moment violations, although it generally changes the alternative marginal. The repair mechanism therefore exploits a feature of the DRO problem that would not generally be available when both marginals are fixed, as in the transportation problems of \citep{gozlan2015kantorovichdualitygeneraltransport,carlier2026weak}. On the other hand, by this difference, the duality techniques of the transportation problems are not obviously helpful for our setting. 

\subsection{Contributions}
We list our contributions here:

\begin{itemize}

\item
\emph{Repair-based normality for CM-OT.}

\begin{itemize}
    \item
    We establish primal-dual equality for CM-OT  (\cref{prop:main_result}) under three complementary conditions (\cref{assn:**a,assn:**b,ass:local-conditional-ui}) that respectively provide budget slack, conditional-moment repairability, and control of positive-loss tails.
    These results/conditions specialize to canonical models in simpler settings (\cref{prop:finite-H-support}), including Wasserstein-, $\phi$-, and Sinkhorn-DRO.
    
    \item
    We give an intermediate implication (\cref{lem:v1=v2}) and checkable sufficient conditions (\cref{lem:**a_implies_**b,lem:subcost-lcui}) for verifying these assumptions.
    
    \item
    We provide examples that clarify the scope of the duality.
    Duality may fail when repairability is absent (\cref{ex:duality_gap_a-not-b}), the repair conditions need not imply classical interiority (\cref{ex:not_slater}), and neither primal nor dual attainment is required for value equality (\cref{ex:nonexistence_primal,ex:nonexistence_dual,ex:nonexistence_dual_Cb}).
\end{itemize}

\item
\emph{Explicit DRO duals and applications.}

\begin{itemize}
    \item
    We decompose CM-OT duality (\cref{thm:combined}) into three steps: primal-dual equality (\cref{prop:main_result}),  interchangeability (\cref{prop:ip}), and a restriction of the dual multiplier space (\cref{thm:Cb_Linfty}).
    
    \item
    We use the CM-OT framework to study a duality for a broad class of problems (\emph{observation-calibrated $\phi$-divergence DRO}, \cref{sec:: Applications}) and obtain a corresponding duality (\cref{cor:observation-phi-duality}), showing that the verification (Assumptions~\ref{ass:1phi}, \ref{ass:2phi}, and \ref{assn:observation-local-tail}) can be checked readily for a number of concrete instances (Lemmas~\ref{lem:1phi}, \ref{lem:2phi}, and \ref{lem:observation-local-a3}, summarized in \cref{tab:observation-local-tail}).
\end{itemize}

\end{itemize}

\section{Preliminaries}\label{sec:preliminaries}
Here we review the classical convex perturbation duality framework \cite{Rockafellar.70,ponstein1980approaches,Ekeland.99,zalinescu2002convex} from a maximization perspective and set some notation.

\paragraph{Notation and Conventions}
Let $\Reals\coloneqq\R\cup\{\pm\infty\}$ and 
$\R_+\coloneqq[0,+\infty)$.
For a normed linear space $X$, write $X^*$
for its continuous dual and $\langle\cdot,\cdot\rangle$ for the pairing.
For convex $g:X\to\Reals$, we use the standard notation  domain $\dom g$, closure $\cl g$, and conjugate $g^*$; ``proper'' means $g> -\infty$ and
$\dom g\neq\varnothing$.
For concave $f: X \to \Reals,$ these same notations will refer to the counterpart, concave notions, and $f^\circ\coloneqq-(-f)^*(-\,\cdot\,)$.
For convex $g$ (resp. concave $f$), $\partial g$ (resp. $\partial f$) denotes the set of subgradients (resp. supergradients).
We also write $\epi g$ and $\hypo f$ for the epigraph of $g$ and the hypograph of $f$, respectively.
The symbols $\iota_A$ and $\ind_A$ denote the $0$--$\infty$ and
$1$--$0$ indicators, respectively, and $[g]_\pm\coloneqq\max\{\pm g,0\}$, dropping the brackets where clear.
For a measurable space $\X$, let $\mathcal P(\X)$ and $\mathcal M(\X)$ denote its probability and finite signed measures. For a  finite measure $\mu$, $L^1(\mu)$ and $L^\infty(\mu)$ carry their usual norms, with $(L^1(\mu))^*=L^\infty(\mu)$ \citep[Theorem~4.14]{brezis2011functional}; for a topological space $\Xi,$ $\mathcal C_b(\Xi)$ denotes the bounded, real-valued, continuous functions over $\Xi$.
We write ${\rm d}_{\rm TV}(\mu,\nu)\coloneqq \sup_A\abs{\mu(A)-\nu(A)}$.

\paragraph{Perturbation Duality}
Given a normed linear (decision) space\footnote{The framework described here extends to Hausdorff locally convex topological vector spaces \citep{Ekeland.99}, with $X^*$ and $U^*$ interpreted as continuous duals. We restrict attention to normed spaces for simplicity.
} $X$ and a concave (primal objective) function $f:X \to \Reals$, we refer to the value $\sup_{x\in X}f(x)$ 
as a (\emph{primal}) concave optimization problem.\PaperFootnote{We adopt the convention of having an optimization problem be made synonymous with its optimal value in $\Reals$ for the sake of expediency. We refer the interested reader to discussions in \citet[Sections~28-29] {Rockafellar.70} that clarify possible misunderstandings and the technicalities around defining a ``problem" in this framework.}
The perturbation framework embeds this value in a family of optimization problems whose values can be described by the sensitivity of the original problem to perturbations in the original problem's data.

Explicitly, this is done by specifying a normed (perturbation) space $U$ and a concave function $F:U\times X\to\Reals$  satisfying $F(0,x)=f(x)$ for all $x\in X$, which is referred to as a \emph{(primal) bifunction}\footnote{``This is not so much a new concept as a different way of treating an old concept, the distinction between `variables' and `parameters.'" \citep[p.~291]{Rockafellar.70}}.
Each choice of $U$ and $F$ yields both a concave \emph{perturbation function} $p: U \to \Reals$ and convex \emph{dual bifunction} $F_{\sf d}: X^* \times U^* \rightarrow \Reals$ via
\[
p(u) \coloneqq \sup_{x \in X} F(u,x), \;\;  F_{\sf d}(x^*,u^*) \coloneqq -F^\circ(u^*,-x^*)
=
\!\!\sup_{x\in X,u\in U} F(u,x) + \inner{x^*}{x} - \inner{u^*}{u}.
\]
In turn, $F_{\sf d}$ yields a convex \emph{dual perturbation function} $q: X^* \rightarrow \Reals$, written
\[
q(x^*) \coloneqq \inf_{u^* \in U^*} F_{\sf d}(x^*,u^*).
\]
Then
\begin{equation} \label{eq::vanillaDual}
p(0) = \sup_{x \in X} f(x) = \sup_{x \in X} F(0,x)\qquad  \text{ and }\qquad q(0) = \inf_{u^* \in U^*} F_{\sf d}(0,u^*) 
\end{equation}
are, respectively, the \emph{primal} and \emph{dual} problems under this perturbation scheme, satisfying $p(0) \leq q(0)$. 
Thus a dual problem is determined by the choice of primal perturbations.

We briefly demonstrate how familiar duals in DRO can be obtained in this way.

\begin{example}[Moment-Constrained DRO]\label{example: Canonical}
Let $\Xi$ be a measurable set, $\ell:\Xi \to \reals$ a measurable loss function, $b \in \reals^m$, and $C:\Xi\to \reals^m,$ a vector of measurable functions. 
Moment-based DRO takes the primal form $\sup_{\mu\in\mathcal P(\Xi)}
    \left\{
        \mathbbm E_\mu[\ell(X)] :
        \mathbbm E_\mu[C(X)]\le b
    \right\}.$ Letting $U = \reals^m$ and
    $F(u, \mu) = \mathbbm E_\mu[\ell(X)] - \iota_{\mathcal{P}(\Xi)}(\mu)-\iota_{\leq b + u}(\mathbbm E_\mu[C(X)]),$
    then $q(0) = \inf_{u^*\in\mathbbm R_+^m}
    \left\{
        (u^*)^\intercal b
        +
        \sup_{x\in\Xi}
        \bigl[\ell(x)-(u^*)^\intercal C(x)\bigr]
    \right\}$. 
\end{example}

\begin{example}[2nd-order Moment DRO {\cite{delage2010distributionally}}]
\label{example:delage-ye}
Let $\Xi\subseteq\R^d$, $\mu_0\in\Xi$, $\Sigma_0\succ0$, 
$\gamma_1\geq0$, and $\gamma_2>0$. The primal is 
$\sup
\{
\mathbbm E_\mu[\ell(X)] : \mu \in \mathcal{P}(\Xi), \; \mathbbm{E}_\mu[C(X)] \in \mathcal{M}\}$, where $C(x)\coloneqq
\bigl(x-\mu_0,(x-\mu_0)(x-\mu_0)^\intercal\bigr)$ and $\mathcal M\coloneqq \{z \in \reals^d: z^\intercal\Sigma_0^{-1}z\leq\gamma_1\} \times \{Z \in \mathbb S^d: Z\preceq\gamma_2\Sigma_0
\}$.
Letting $U=\R^d\times\mathbb S^d$,  $F(u,\mu)
=
\mathbb E_\mu[\ell(X)] - \iota_{\mathcal{P}(\Xi)}(\mu)
-\iota_{\mathcal M+u}\bigl(\mathbbm E_\mu[C(X)]\bigr)$, the corresponding 
dual 
is 
\begin{align*} q
(0) 
&= \inf_{(u_1^*,u_2^*)\in\reals^d\times\mathbb S^d} \sup_{\mu\in\mathcal P(\Xi),\,m\in\mathcal M} \left\{ \mathbbm E_\mu[\ell(X)] -\left\langle u^*,\mathbbm E_\mu[C(X)]-m \right\rangle \right\}\\ 
&= \inf_{(u_1^*,u_2^*)\in\reals^d\times\mathbb S_+^d} \left\{ \sigma_{\mathcal M}(u_1^*,u_2^*) + \sup_{x\in\Xi} \left[ \ell(x) -(u_1^*)^\intercal(x-\mu_0) -u_2^*\bullet (x-\mu_0)(x-\mu_0)^\intercal \right] \right\}, 
\end{align*}
where $\sigma_\mathcal{M}(u^*_1,u^*_2)$ can be written explicitly as $\sqrt{\gamma_1}\,
\|\Sigma_0^{1/2}u_1^*\|
+\gamma_2(\Sigma_0\bullet u_2^*)$. \cite{delage2010distributionally}'s form follows exactly upon substituting $Q\coloneqq u_2^*$ and $q\coloneqq u_1^*-2Q\mu_0,$ expanding the centered quadratic, and epigraphing the inner supremum with $r$ and the remaining objective terms with $t.$ 
\end{example}

Although the above dual problems can certainly be derived without invoking perturbations, an  advantage of the framework emerges when designing and verifying regularity conditions for primal-dual equality.
Towards that end, we rephrase some classical results below, in which compactness and interiority play complementary roles in ensuring $p(0)=q(0)$.

\begin{proposition}[\citet{Ekeland.99}]
    \label{prop:convex_duality}
    In the above setting, if $p$ is proper and $p(0) \in \reals$, then the following hold:
\begin{enumerate}[parsep=0pt,topsep=2mm]
    
    \item [{\rm (a)}]
    \setitemlabel{(a)}
    \label{convex_duality:normality}
    \emph{Normality  (ib.~Prop.~III.2.1):} 
    $p(0) = q(0)$
    $\Leftrightarrow$ 
    $p(0) = (\cl p)(0)$ $\Leftrightarrow$
    $\displaystyle p(0) \geq \limsup_{u \rightarrow 0}p(u)$;

    \begin{enumerate}[parsep=0pt,topsep=0pt]
        \item[{\rm (a$^*$)}]
        \setitemlabel{(a$^*$)}
        \label{cond:compact}
        \emph{Compactness}
        \emph{(cf.~\cite[Theorem 6.3.2]{berge1963topological}):}
        If $F$ is jointly upper semicontinuous and $\dom F\subseteq U\times K$ for some compact $K\subseteq X$, then $p(0)=q(0)$ and $\argmax_{x\in X}F(0,x)\neq\varnothing$.
    \end{enumerate}

    \vspace{2mm}
    \item [{\rm (b)}]
    \setitemlabel{(b)}
    \label{convex_duality:stability}
    \emph{Stability (ib.~Prop. III.2.2):}
    $\partial p(0) \neq \varnothing \Leftrightarrow  [p(0) = q(0) \text{ and } \displaystyle \argmin_{u^* \in U^*} F_{\sf d}(0,u^*) \neq \varnothing] $;
    \begin{enumerate}[topsep=0pt]
        \item[{\rm (b$^*$)}]
        \setitemlabel{(b$^*$)}
        \label{cond:slater}
        \emph{Interiority (ib.~Prop.~I.2.5, Prop.~I.5.2):}
        If $0 \in \intr \dom p$ and $p$ is finite and continuous at some $u \in U$, then $\partial p(0) \neq \varnothing$.

    \end{enumerate}

    \end{enumerate}
\end{proposition}
We now recall some familiar applications of \cref{prop:convex_duality} to DRO.

\begin{example}[Compactness in Decisions] \label{example: CompactnessInDecisions}
\Cref{prop:convex_duality}\ref{convex_duality:normality} guides the design of conditions for equality in \cref{example: Canonical}
     with focus on closing $p$ at $0$.
     In \cref{example: Canonical}, if $\Xi$ is a compact metric space, $C$ is continuous, $\ell$ is upper semicontinuous, and the primal is feasible, i.e., $p(0) > -\infty$, then $p$ is closed; indeed, $\hypo  F = \{(\mu, u, t): t \leq \mathbb{E}_\mu[\ell(X)], \; \mathbb{E}_\mu[C(X)] \leq b + u, \; \mu \in \mathcal{P}(\Xi)\}$ is closed, so upon applying the projection mapping $(\mu, u, t) \to (u,t)$, we attain $\hypo  p$, which must be closed since $\mathcal{P}(\Xi)$ is weakly compact when $\Xi$ is strongly compact.
     This is the same inf-compactness route to value-function closedness used in conic duality (cf.~\citet[Prop.~2.4]{shapiro2001duality}).
\end{example}

\begin{example}[Interiority Conditions]
\label{example:Slater's for DRO}

The hypothesis of \cref{prop:convex_duality}\ref{cond:slater} is equivalent to the existence of a neighborhood $N\ni0$ and $M > -\infty$ such that $p(u)=\sup_x F(u,x) > M$ for all $u \in N$.
The classical Slater's condition verifies this with a distinguished decision $x^0$, whereas Slater-type conditions verify this by range conditions (\cite[(2.20)-(2.23) and Section~3]{shapiro2001duality} and 
\cite[Proposition~2.1]{xu2018distributionally}).

The interiority condition $0\in\intr\dom p$ alone, or its generalization $0\in\ri\dom p$ suffices when $U$ is finite dimensional \citep[Theorem 23.4]{Rockafellar.70}, since a proper concave function is continuous on the interior of its domain.
When $U$ may not be finite dimensional, the interior can be replaced by various notions of generalized interior \citep{borwein1992partially,shapiro2001duality,bot2012regularity,vancuong2023generalized}.

In \cref{example: Canonical}, assuming $p(0) \in \reals,$ if there exists $\bar{\mu}\in \mathcal{P}(\Xi)$ such that $\mathbb{E}_{\bar\mu}[C(X)] < b$ and $\mathbb{E}_{\bar\mu}[\ell(X)] > -\infty$, then $0 \in \intr\dom p$. 
In \cref{example:delage-ye}, assuming $p(0) \in \reals,$ if $\gamma_1 > 0$, then $\mathbb{E}_{\delta_{\mu_0}}[C(X)] = (0,0) \in \intr \mathcal{M}$ so that $(0,0) \in \intr \dom p$. If instead $\gamma_1 = 0$ but $\mu_0 \in \ri \Xi,$ then $p(u) > -\infty$ for all sufficiently small $u=(u_1,u_2)\in \operatorname{lin}(\Xi - \mu_0)\times\mathbb S^d = \aff(\dom p)$; indeed, $\mathbb E_{\delta_{\mu_0+u_1}}[C(X)] - u
    =(0,u_1u_1^\intercal-u_2)\in\mathcal M$ for sufficiently small $u$, and thus $(0,0)\in\ri\dom p$.
\end{example}

These examples illustrate two familiar routes to primal-dual equality, yielding either primal or dual attainment.
These routes are generally sufficient for achieving the requisite closedness of $p$, although they may not be always (practically) available.
In the following sections, we target the characterization of normality directly by constructing feasible approximations to nearly feasible couplings and controlling the resulting change in objective value and provide sufficient conditions that may be readily verified in terms of the problem data without guaranteeing either primal or dual attainment.

\section{CM-OT and the Duality Question}
\label{sec:lacking_interior}

The CM-OT framework of \citet{jiajin2025unifying} unifies transport-based and $\phi$- divergence models by accounting for distributions that simultaneously perturb likelihood ratios and outcomes of the reference distribution $\hat\mu$.
Here we recall that setting, its duality, and related context.

\subsection{The CM-OT Model}
\label{sec:cmot-problem-setting}
We now introduce the setting of \citet{jiajin2025unifying} and set several conventions.
Let $\mathcal{V}$ be a topological space equipped with its Borel sigma algebra $\mathcal{G}$, and let
$\W\subseteq\R$ be a nonempty closed set carrying its subspace topology and equipped with
Borel $\sigma$-algebra $\mathcal B$. Set $\U\coloneqq\V\times\W,$ and $\mathcal F\coloneqq\mathcal G\otimes\mathcal B.$
For $u\in\U$, write $u=(v,w)$. Let $\mathcal Z$ be a Polish space equipped with its
Borel $\sigma$-algebra $\mathcal H$, and let
$H:(\V,\mathcal G)\to(\mathcal Z,\mathcal H)$ be a
measurable information map. We fix a reference measure $\hat\mu\in\mathcal P(\U)$
with the $\V$-marginal 
$\hat\nu \coloneqq [(v,w)\mapsto v]_\# \hat\mu$
and write $
    \hat\zeta\coloneqq H_\#\hat\nu.$
For a coupling $\gamma$ with first marginal $\hat\mu$,
we write $(\hat V,\hat W,V,W)\sim\gamma$ and
$\hat Z\coloneqq H(\hat V) \sim \hat\zeta$. %
We also write $\Gamma(\alpha,\beta)$
for the couplings of $\alpha,\beta\in\mathcal P(\U)$ and
$\Gamma_{\hat\mu}\coloneqq\bigcup_{\mu\in\mathcal P(\U)}
\Gamma(\hat\mu,\mu)$. For $\gamma\in\Gamma_{\hat\mu}$ we write
$(\hat V,\hat W,V,W)\sim\gamma$ with the understanding that $\hat U \equiv (\hat V, \hat W) \sim \hat \mu$ and $U \equiv(V,W) \sim \mu$ for some $\mu \in \mathcal{P}(\U)$.

We consider an optimal-transport discrepancy between members of $\mathcal{P}(\U)$, building on \citet{jiajin2025unifying}.\footnote{Related fixed-marginal transport problems with moment constraints are studied by \citet{carlier2026weak}.}
This will incorporate an $(\mathcal{F} \times \mathcal{F})$-measurable transport cost 
$c:\mathcal U \times\mathcal U \to (-\infty,+\infty]$ with $\inf c > -\infty$, 
a conditional-moment target function $h\in L^1(\hat\zeta)$,
and an $\mathcal{F}$-measurable objective function
$f:\mathcal U\to[-\infty,\infty)$.
We also define $\breve\Gamma_{\hat\mu}\coloneq\{\gamma\in\Gamma_{\hat\mu} : \mathbbm{E}_\gamma[\abs{W}] < \infty,\;\mathbbm{E}_\gamma[c]<+\infty\}$.

Finally, we adopt $0\cdot\infty=0$ and remark that
under the assumptions of our duality results below, all expectations are well-defined without invoking inf- or sup-addition conventions \citep[p.~15]{rockwets}.\footnote{%
Indeed, for every $\gamma\in\breve\Gamma_{\hat\mu}$, both $W$ and $c$ are integrable, so $\mathbbm E_\gamma[f_+]<+\infty$ ensures that the expectation defining
the objective
in $\dualPsiRef$ is unambiguous for every $\lambda\geq0$ and $\psi\in L^\infty(\hat\zeta)$.
Moreover, if there exists $\gamma^0\in\breve\Gamma_{\hat\mu}$ with $\mathbbm E_{\gamma^0}[|f|]<+\infty$, then the supremum appearing in 
the objective of
$\dualPsiIPRef$ admits an integrable lower bound under $\hat\mu$, so its expectation is also unambiguous whenever measurable.
Both integrability conditions are ensured by \cref{assn:**a,ass:local-conditional-ui} below.
} %
In the interim, for a probability measure $\mu$ and a measurable
extended-real-valued function $g$, expectations are interpreted
as lower integrals \citep[p.~675-676]{rockwets}, i.e.,
\(
    \mathbbm{E}_\mu[g]
    \coloneqq
    \int g_+\,d\mu-\int g_-\,d\mu
    \mbox{ and }
    +\infty-(+\infty)\coloneqq-\infty,
\)
where the integrals of $g_\pm$ are the usual nonnegative
Lebesgue integrals using the sup-addition convention.

\begin{definition}[CM-OT Discrepancy \cite{jiajin2025unifying}]
    \label{def:OT_discrepancy}
    In the setting above, 
    we define the resulting \emph{optimal transport discrepancy with conditional moment constraints}, or \emph{CM-OT discrepancy}, by
    \begin{align}
        \label{eq:M_h}
        \mathbbm{M}_h(\hat\mu,\mu)
        \coloneqq
        \left\{\!
        \begin{array}{cl}
            \displaystyle 
            \inf_{\gamma\in\Gamma(\hat\mu,\mu) \cap \breve\Gamma_{\hat\mu}} &
            \mathbbm{E}_\gamma[c(\hat V,\hat W;V,W)]\\
            \st & \mathbbm{E}_\gamma[W\mid\hat Z]-h(\hat Z)=0,\;{\hat\zeta\text{-a.s.}}
        \end{array}
        \!\right\} \quad \forall \hat\mu, \mu \in \mathcal{P}(\U).
    \end{align}    
    \end{definition}
    \begin{remark}
        \citet{jiajin2025unifying} investigated duality when $h \equiv 1$ and $H$ is the identity map.  
    \end{remark}
\noindent The discrepancy $\mathbbm{M}_h(\hat\mu,\mu)$
yields the following uncertainty quantification problem.

\begin{definition}[CM-OT (Primal) Problem \cite{jiajin2025unifying}]
    For $\rho\in\R$ and $\hat\mu\in\mathcal{P}(\U)$,
    the \emph{(primal) conditional-moment and optimal-transport-constrained} uncertainty quantification problem, or \emph{CM-OT problem}, is given by
    \begin{equation}
        \label{eq:primal}
        \sup \bigl\{ \mathbbm{E}_{\mu}[f] : \mu\in\mathcal{P}(\U), \;\mathbbm{M}_h(\hat\mu,\mu)\leq\rho \bigr\}.
        \tag{$\primal$} \;\;\;
    \end{equation}
\end{definition}
\citet{jiajin2025unifying} demonstrate that \eqref{eq:primal} unifies several standard uncertainty quantification models that employ a variety of discrepancies, including: 
generalized $\phi$-divergence
\citep{bental2011robust,agrawal2021optimal},
Sinkhorn \citep{wang2025sinkhorn},
and 
traditional Wasserstein optimal transport
\citep{murthy2019quantifying,gao2022short}.
Moreover, it is noted that \eqref{eq:primal} shares similarities with martingale optimal transport \citep{zhou2021martingale,li2022martingale}. 

\subsection{CM-OT Duality}
\label{sec: Conjecture}
Let $\Psi \subseteq \R^\mathcal{Z}$ be a set of $\mathcal{H}$-measurable, real valued functions defined on $\mathcal{Z}$.
The following was proposed and studied by \cite{li2022martingale} and \cite{jiajin2025unifying} as a dual problem   for \eqref{eq:primal}:
\begin{equation}
\displaystyle\inf_{{\lambda\geq0,\,\psi\in\Psi}} \;\; \lambda \rho +
    \mathbbm{E}_{\hat\mu}\Bigl[\;
    \sup_{u \in \U} 
    f(u) - \psi(\hat Z)\cdot (w-h(\hat Z))  - \lambda c(\hat U; u)
    \;\Bigr]. \tag{$\pcw_\Psi+\ip$}\label{eq:dual_Psi_IP}
\end{equation}
As in transport DRO, the duality equality $[\eqref{eq:primal} = \dualPsiIPRef]$ is more than just a convex-analytic concept. Indeed,
\begin{align}
    \dualPsiAnchor
    \eqref{eq:primal}
    &\leq
    \displaystyle\inf_{{\lambda\geq0,\,\psi\in\Psi}} \;\; \lambda \rho +
    \sup_{\gamma \in \breve\Gamma_{\hat\mu}}
    \;\; \mathbbm{E}_\gamma [
    f(V,W) - \psi(\hat Z)\cdot (W-h(\hat Z))  - \lambda c(\hat U; U)
    ]
    \tag{$\pcw_\Psi$}\label{eq:dual_Psi}
    \\
    \dualPsiIPAnchor
    &\leq\; \eqref{eq:dual_Psi_IP} \nonumber,
\end{align}
where
$\dualPsiRef$ is a convex dual problem to $\eqref{eq:primal}$, and $\dualPsiIPRef$ reflects an interchange of the 
$\sup$ and expectation operations in $\dualPsiRef$.
Hence, for a choice of $\Psi$, the \emph{CM-OT duality} $[\eqref{eq:primal} = \dualPsiIPRef]$
is equivalently the combination of two relations:
\begin{enumerate}[itemsep=1mm,topsep=1mm]
    \item[\textlabel{cond:dro_zerogap}{\textnormal{I.}}]
    $[\eqref{eq:primal} = \dualPsiRef]$ Normality; and
    \item[\textlabel{cond:dro_IP}{\textnormal{II.}}] $[\dualPsiRef = \dualPsiIPRef]$  Interchangeability Principle (IP) (e.g.,
    \cite{rockafellar1971integral,gao2022short}).
\end{enumerate} 

Existing results illustrate two familiar routes to CM-OT duality. \citet{li2022martingale}, for the case of finite-supported $\hat\mu$, identity map $H$, $\Psi = \reals^{\operatorname{supp}(\hat Z)}$, and $h \equiv 1$, pursue interiority: the conditional moment constraint is represented as a finite-dimensional perturbation, so a Slater-type condition can yield normality (in fact, stability), with interchangeability holding trivially. On the other hand, \citet{jiajin2025unifying}, for the case of identity map $H$, $\Psi \equiv \mathcal C_b(\Z)$ and $h\equiv 1$, jointly achieve normality and interchangeability with a route centered around compactness: 
\textlabel{assn:U_compact}{\textnormal{(i)}}
$\mathcal{U}$ is compact;
\textlabel{assn:f_closed}{\textnormal{(ii)}}
$f:\U\to\R$ is upper semicontinuous and $f\in L^1(\hat \mu)$; and
\textlabel{assn:c_closed}{\textnormal{(iii)}}
$c:\U\times\U\to(-\infty,+\infty]$ is lower semicontinuous and $c(u, u) = 0$ for all $u \in \mathcal{U}$. 

\paragraph{Beyond Interiority and Compactness} \label{sec: Counterexample}

Neither route, however, is entirely satisfactory for the general CM-OT setting considered here. The conditional moment constraint induces function-valued perturbations and multipliers. Requiring interiority in such a perturbation space amounts to requiring feasibility for a full neighborhood of conditional-moment perturbations, which can be substantially stronger than what is needed. In particular, the ability to correct a moment violation may vary with the information state and may become arbitrarily small on parts of the state space, even though sufficiently small violations remain repairable in an asymptotically negligible manner. Although assuming the reference measure $\hat \mu$ is finite-supported, as in \cite{li2022martingale}, greatly simplifies the interiority route, this can also be overly restrictive. Meanwhile, compactness excludes many natural DRO models defined over an unbounded support space $\mathcal{U}$ and is particularly restrictive when one wishes to accommodate unbounded loss functions. 

The preceding discussion motivates looking beyond interiority and compactness. Relaxing the latter requires particular care. As conjectured by \citet{jiajin2025unifying}, compactness cannot simply be removed from the hypothesis of their duality theorem (ib. Theorem 4.2), i.e.,
$[\neg\ref{assn:U_compact} \land \ref{assn:f_closed} \land \ref{assn:c_closed}] \nRightarrow [\eqref{eq:primal} = \dualPsiRef]$, as our following example confirms and resolves for the authors of \cite{jiajin2025unifying}.

\begin{example}[Duality Gap without Compactness]
    \label{lem:duality_gap_unbounded_f}
    Let
    $\V\coloneqq[1, +\infty)$ and 
    $\W\coloneqq\R_+$ be endowed with standard topologies, $H=\operatorname{id}_{\V}$, 
    $\mathcal{F}$ be the product of their respective Borel $\sigma$-fields,
    $f(v,w)\coloneqq v\cdot(w-1)$,
    $h(\hat V)\equiv 1$,
    $\rho>0$,
    and $c(\hat v,\hat w;v,w)\coloneqq \iota_{\Delta}(\hat v, v) + (w - \hat w)^2,$ where $\Delta \coloneqq \{(\hat v, v): v = \hat v\}$.
    Let $\hat V \sim \hat\nu$ with $\mathbb{E}_{\hat\nu}[\hat V] < +\infty$ and $\mathbb{E}_{\hat\nu}[\hat V^2] = +\infty$. Also let $\hat W \equiv 1$. Then $(\hat V, \hat W) \sim \hat\mu = \hat \nu \otimes \delta_1$ and $\hat\zeta = \hat\nu.$ Finally let $\Psi\coloneqq L^\infty(\hat\nu)$. 
    
    Then  \eqref{eq:primal} is clearly feasible, and, further, $0=\eqref{eq:primal}$. 
    Indeed, any finite-discrepancy feasible alternative $\mu$ admits a coupling $(\hat V, \hat W, V, W) \sim \gamma$ of $\hat\mu$ and $\mu$ such that $V = \hat V$, $\gamma-a.s.$ as well as $\mathbbm{E}_\gamma[W\; |\; \hat V] = 1$ so that 
    \[
        \mathbbm{E}_\gamma[\hat VW]= \mathbbm{E}_{\hat\nu}\bigl[\hat V \cdot\mathbbm{E}_{\gamma}[W\; |\; \hat V]\bigr] = \mathbbm{E}_{\hat\nu}[\hat V] < +\infty
    \]
    yields
    \[
    \mathbbm{E}_{\gamma}[f] = \mathbbm{E}_\gamma[\hat VW] - \mathbbm{E}_{\hat\nu}[\hat V] = 0.
    \]
    
    On the other hand, we can show that 
    \(
        \dualPsiRef
        =
        +\infty.
    \)
    Indeed, any dual decision $(\lambda, \psi) \in [0, +\infty) \times L^\infty(\hat\nu)$ leads to a dual objective value of $+\infty.$ Towards showing this, let $\psi\in L^\infty(\hat\nu)$, and we will show 
    \[
    \sup_{\gamma \in \breve\Gamma_{\hat\mu}}
    \;\; \mathbbm{E}_\gamma [
    V(W-1) - \psi(\hat Z)\cdot (W-1)  - \lambda c(\hat U; U)] = +\infty
    \]
    for the separate cases of $\lambda = 0$ and $\lambda > 0.$ For the case of $\lambda = 0,$ we note that
    \begin{align*}
    \sup_{\gamma\in\breve\Gamma_{\hat\mu}}
    \mathbbm{E}_\gamma[(V - \psi(\hat V)) \cdot (W-1)]
    \geq
    \sup_{r\geq1} 
    (r-1) \cdot \hat\nu\left(\hat V > \|\psi\|_{L^\infty(\hat\nu)} + 1\right)
    =
    +\infty, 
    \end{align*}
    since for any integer $r\geq 1,$ we can define $W^r \coloneqq r\ind_{D} +  \ind_{D^c}$, with $D \coloneqq \{\hat V>\|\psi\|_{L^\infty(\hat\nu)}+1\}$ being a set of positive $\hat\nu$ measure by $\mathbbm{E}_{\hat\nu}\left[\hat V^2\right] = +\infty$ , so that 
    $(\hat V, \hat W \equiv 1, V^r \equiv \hat V, W^r)\sim
    \gamma_r\in\breve\Gamma_{\hat\mu}$, and 
    \(
    \mathbbm{E}_{\gamma_r}[(\hat V - \psi(\hat V))\cdot (W^r-1)]
    \geq
    (r-1)\cdot \hat\nu(D). 
    \)
    For the case of $\lambda > 0,$ define $W^r \coloneqq 1 + \min\left(\frac{[\hat V - \psi(\hat V)]_+}{2\lambda}, r\right)$ so that the random vector $(\hat V, \hat W \equiv 1, V^r \equiv \hat V, W^r)\sim
    \gamma_r\in\breve\Gamma_{\hat\mu}$. Upon using the shorthand $X_r \coloneqq \min\left(\frac{[\hat V - \psi(\hat V)]_+}{2\lambda}, r\right)$, we again find a divergence as desired, since
    \begin{align*}
    &\mathbbm{E}_{\gamma_r}[(V^r - \psi(\hat V)) \cdot (W^r-1) - \lambda c] = \mathbbm{E}_{\gamma_r}\bigl[(\hat V - \psi(\hat V))X_r - \lambda X_r^2\bigr]\\
    &\geq \frac{1}{4\lambda}\mathbbm{E}_{\hat\nu}\left[\min\left(([\hat V - \psi(\hat V)]_+)^2, 4\lambda^2r^2\right)\right] \underset{r \rightarrow \infty}{\longrightarrow} \frac{1}{4\lambda}\mathbbm{E}_{\hat\nu}\left[([\hat V - \psi(\hat V)]_+)^2\right] = +\infty,
    \end{align*} 
    where the final equality holds because 
    \begin{align*}
    \mathbbm{E}_{\hat\nu}\left[(\hat V - \|\psi\|_{\infty})^2; [\hat V > 2\|\psi\|_{\infty}]\right] \geq \mathbbm{E}_{\hat\nu}\left[(\hat V/2)^2; [\hat V > 2\|\psi\|_{\infty}]\right] = +\infty.
    \end{align*}    
\end{example}
\noindent Although \citet{jiajin2025unifying} also propose consideration of an integrability and growth condition on $c$, namely, that there are constants $\delta\in (0,1]$ and $k > 0$ and a Borel function $\varphi: \W \rightarrow \reals_+$ such that
\[
\mathbb{E}_{\hat\mu}\bigl[\varphi(\hat W)\bigr] < +\infty \quad \text{ and } \quad c((v,w); (\hat v, \hat w)) \geq k \|w\|^{1+\delta} - \varphi(\hat w) \;\; \forall (v,w), (\hat v, \hat w) \in \U,
\]
the above \Cref{lem:duality_gap_unbounded_f} in fact satisfies this growth condition with $\delta = 1, k = 1/2, \varphi(\hat w) = \hat w^2$, thus demonstrating that this condition cannot by itself provide duality in the absence of compactness. Finally, we remark that \Cref{lem:duality_gap_unbounded_f} features an upper semicontinuous loss $f$ that is unbounded over $\U$, which is possible when $\U$ is no longer assumed compact; however, as we'll see, unbounded loss by itself does not break duality.

In the following sections, we formalize our new conditions for CM-OT duality beyond interiority and compactness. Seeing as how the two familiar routes are not required per se, we tailored our conditions more directly to the CM-OT structure.
Our conditions provide a repair mechanism so that any sequence of nearly feasible couplings can be made feasible while incurring an asymptotically negligible drop in objective value. 

\section{Repair-Based Conditions for \texorpdfstring{$[\eqref{eq:primal} = \dualPsiRef]$}{[(P) = (D-Psi)]}}
\label{sec:Closing Duality Gap}

Next we introduce our repair-based conditions.
For notational simplicity, 
let the conditional-moment
residual $A\gamma\in L^1(\hat\zeta)$ for $\gamma\in\breve\Gamma_{\hat\mu}$, %
be a measurable version satisfying $(A\gamma)(\hat Z)
    =
    \mathbb{E}_\gamma[W\mid\hat Z]-h(\hat Z), \;\; \hat\zeta \text{-a.s.}$ with $\inner{A\gamma}{\psi} = \mathbbm{E}_\gamma[\psi(\hat Z)\cdot (W-h(\hat Z))]\in\reals$ when $\psi\in L^\infty(\hat\zeta)$.
\begin{assumption}
    \label{assn:**a}
    There exist $(\hat V, \hat W, V^0, W^0) \sim \gamma^0 \in \breve\Gamma_{\hat\mu}$%
    s.t.
    \begin{align*}
       \mathbbm{E}_{\gamma^0}[c] < \rho,\quad
        A\gamma^0=0,
        \;\;\hat\zeta \text{-a.s.}
        \quad\text{and}\quad
        \mathbbm{E}_{\gamma^0}[f_-(V^0, W^0)] <+\infty.
    \end{align*}
\end{assumption}
\begin{assumption}
    \label{assn:**b}
    There exist $\gamma^+,\gamma^-\in
\breve\Gamma_{\hat\mu}$, %
with $(\hat V,\hat W,V^\pm,W^\pm)\sim\gamma^\pm$  
    such that
    \begin{align*}
        A\gamma^+>0,
        \;\;
        A\gamma^-<0,
        \;\; \hat\zeta\text{-a.s.},
        \;\; \mathbbm{E}_{\gamma^\pm}[c] < +\infty,
        \;\;\text{and}\;\;
        \mathbbm{E}_{\gamma^\pm}[f_-(V^\pm, W^\pm)]<+\infty.
    \end{align*}
\end{assumption}
\noindent For $\delta>0$, let
$\mathcal{N}_{\delta}
    \coloneqq
    \bigl\{
        \gamma\in\breve\Gamma_{\hat\mu}:
        \mathbb{E}_{\gamma}[c]\leq \rho+\delta,
        \;
    \|A\gamma\|_{L^1(\hat\zeta)}\le\delta
    \bigr\}$.
For each $\gamma\in\breve\Gamma_{\hat\mu}$ %
let $f_\gamma^+: \Z \to [0, +\infty]$ be a measurable version satisfying\\
\(
    f_{\gamma}^+(\hat z)
    \coloneqq
    \mathbb{E}_{\gamma}
    [[f]_+(V,W)\mid \hat Z=\hat z].
\) 
\begin{assumption}[Local conditional uniform integrability]
\label{ass:local-conditional-ui}
There exists $\delta>0$ such that the family of conditional expectation functions $\left\{
        f_\gamma^+:
        \gamma\in\mathcal{N}_{\delta}
    \right\}$
is uniformly integrable in $L^1(\hat\zeta)$. That is,
\[
    \sup_{\gamma\in\mathcal{N}_{\delta}}
    \mathbbm{E}_{\hat \zeta}\bigl[
        f_\gamma^+(\hat Z)\bigr]
    <+\infty
~~\text{ and }~~
    \lim_{\epsilon\downarrow0}
    \;
    \sup_{\gamma\in\mathcal{N}_{\delta}}
    \;
    \sup_{\substack{E\in\mathcal{H}\\
                    \hat\zeta(E)\leq\epsilon}}
    \mathbbm{E}_{\hat\zeta}\bigl[
        f_\gamma^+(\hat Z)\,
        ; \hat Z \in E\bigr]
    =0.
\]
\end{assumption}
\noindent 

 We remark that \cref{assn:**a,assn:**b,ass:local-conditional-ui} can collectively hold even when the classical interiority condition in
\cref{prop:convex_duality}\ref{cond:slater} fails; we refer the reader to \cref{ex:not_slater} in
\PaperSupplementSection{apx:nonexistence} for this point. Assumptions \ref{assn:**a}, \ref{assn:**b}, and
\ref{ass:local-conditional-ui} will now be discussed in turn, from how to verify them, to the role they play in the repair mechanism to come.

\subsection{On \texorpdfstring{\cref{assn:**a}}{Assumption 1}}
\label{sec:on_assn_1}
So long as $\eqref{eq:primal}$ admits a feasible distribution with finite negative-loss integral, then \cref{assn:**a} will hold in the case that $\rho$ is replaced by $\rho+\epsilon$ for any $\epsilon>0$.
This is useful since many works treat
$\rho$ as a tunable parameter representing a decision maker's prescribed level of conservatism
\citep{bayraksan_love_2015,esfahani_kuhn_2018,kuhn_esfahani_nguyen_shafiee_2019,murthy2019quantifying,rahimian_mehrotra_2022,aolaritei_shafiee_dorfler_2026}; hence, in this respect the viability of \cref{assn:**a} (and duality) may also be of consideration to said decision maker. 

One consequence of \cref{assn:**a} is that it affords a reformulation of the feasible region to $\eqref{eq:primal}$ in which the transport cost constraint is now made explicit as a linear inequality constraint. The resulting reformulation \eqref{eq:primal_explicit} has precisely $\eqref{eq:dual_Psi}$ as its convex dual, which reduces $\eqref{eq:primal} = \eqref{eq:dual_Psi}$ to a matter of normality. A similar reformulation is obtained in \cite{jiajin2025unifying} but under a different set of assumptions;
in contrast, we leverage a simple mixing argument.
\begin{restatable}{lemma}{AssumptionOneLemma}
    \phantomsection\label{lem:v1=v2}
    Define the explicit problem over couplings:
    \begin{align}
    \label{eq:primal_explicit}
            \sup\bigl\{  \mathbbm{E}_\gamma[f(V,W)] : 
            \gamma\in
            \breve\Gamma_{\hat\mu},\;
            \mathbbm{E}_\gamma[c]\leq\rho,\;\text{and}\;\;
            \mathbbm{E}_\gamma[W\mid\hat Z]=h(\hat Z),\;\text{$\hat\zeta$-a.s.}
            \bigr\}
        \tag{\primalExplicit}
    \end{align}
    Under \cref{assn:**a}, it holds that $\eqref{eq:primal}=\eqref{eq:primal_explicit}$.
\end{restatable}
\begin{proof}[Proof of \cref{lem:v1=v2}]
   By definition, it is clear that $\eqref{eq:primal} \geq \eqref{eq:primal_explicit}$. Hence, it suffices to show that $\eqref{eq:primal} \leq \eqref{eq:primal_explicit}$. We first consider the case that $\eqref{eq:primal}\in\R$, then secondly consider the case that $\eqref{eq:primal}=+\infty$.
    
     Considering the case $\eqref{eq:primal}\in\R$,  
    let $\epsilon > 0$ be arbitrary. Then let $\mu_\epsilon$ be $\epsilon$-optimal to \eqref{eq:primal}, meaning $\mathbbm{E}_{\mu_\epsilon}[f] > \eqref{eq:primal} - \epsilon$. By the definition of infimum in $\mathbbm{M}_h$, there exists $\gamma_\epsilon\in\Gamma(\hat\mu,\mu_\epsilon)\cap\breve\Gamma_{\hat\mu}$
     satisfying
    \(
    \mathbbm{E}_{\gamma_\epsilon}[W-h(\hat Z)\mid\hat Z]=0,
    \;\hat\zeta\text{-a.s.}
    \)
    and
    \(
    \mathbbm{E}_{\gamma_\epsilon}[c]\leq \rho+\epsilon.
    \)
    By \cref{assn:**a}, $a=\rho - \mathbbm{E}_{\gamma^0}[c]>0$ with
    {$\mathbbm{E}_{\gamma^0}[W^0\mid\hat Z]-h(\hat Z)=0$, $\hat\zeta$-a.s.}
    Construct the mixture
    \(
    \bar\gamma_\epsilon\coloneqq(1-t)\gamma_\epsilon + t\gamma^0,
    \; \text{for}\;
    t\coloneqq \frac{\epsilon}{a+\epsilon}\in(0,1),
    \)
    which satisfies
    \(
    \mathbbm{E}_{\bar\gamma_\epsilon}[W-h(\hat Z)\mid\hat Z]=0,
    \;{\hat\zeta\text{-a.s.}}
    \)
    and
    \(
    \mathbbm{E}_{\bar\gamma_\epsilon}[c]\leq\rho.
    \)
    Then $\bar\gamma_\epsilon$ is feasible to \eqref{eq:primal_explicit}, and computing its objective value, we observe that
    \begin{align*}
        \eqref{eq:primal_explicit}
        \geq
        \frac{a}{a+\epsilon} \mathbbm{E}_{\mu_\epsilon}[f] + \frac{\epsilon}{a+\epsilon}\mathbbm{E}_{\gamma^0}[f]
        >
        \frac{a}{a+\epsilon}(\eqref{eq:primal}-\epsilon) + \frac{\epsilon}{a+\epsilon}\mathbbm{E}_{\gamma^0}[f].
    \end{align*}
    Noting that $\mathbbm{E}_{\gamma^0}[f]>-\infty$ and that $\epsilon>0$ was arbitrary, taking $\epsilon\downarrow0$ yields \eqref{eq:primal_explicit} = \eqref{eq:primal}.

    Considering the case $\eqref{eq:primal}=+\infty$, choose feasible $\mu_{1/n}$ with $\mathbbm{E}_{\mu_{1/n}}[f]\geq n$ and apply the mixing construction above with $\epsilon=1/n$; since $\mathbbm{E}_{\gamma^0}[f]>-\infty$, we again find $\eqref{eq:primal_explicit}=\eqref{eq:primal}$ via
    \[
    \eqref{eq:primal_explicit} \geq 
     \frac{a}{a+1/n} \mathbbm{E}_{\mu_{1/n}}[f] + \frac{1/n}{a+1/n}\mathbbm{E}_{\gamma^0}[f] \geq \frac{(an^2+\mathbbm{E}_{\gamma^0}[f])}{(an+1)}\to+\infty = \eqref{eq:primal}.
    \]
\end{proof}
\noindent

Another consequence of \cref{assn:**a} is that any $\gamma\in\breve\Gamma_{\hat\mu}$ can be mixed with $\gamma^0$ so that the resulting mixture has transport cost strictly less than $\rho$.
In this way, we will be able to repair transport cost violations first before subsequently repairing feasibility for the conditional moment constraint, with the latter not undoing the efforts of the former.  

Finally, we remark that it can be readily seen that \cref{assn:**a} alone is insufficient for obtaining normality; see \cref{ex:duality_gap_a-not-b} in \PaperSupplementSection{apx:A1A2}. 

\subsection{On \texorpdfstring{\cref{assn:**b}}{Assumption 2}}
 \cref{assn:**b} provides reference measures with which an infeasible measure can be mixed to correct violations of the conditional moment constraint. This correction mechanism is central to our proof strategy: it ensures that vanishing perturbations of \eqref{eq:primal_explicit} admit well-behaved feasible approximations. We provide a sufficient condition for verifying \cref{assn:**b} that ties \cref{assn:**a} together with some regularity of transport cost $c$ and objective $f$.
\begin{lemma}[Sufficient condition for \cref{assn:**b}]
\label{lem:**a_implies_**b}
Suppose $\gamma^0$ satisfies Assumption~1. Suppose there exist
measurable functions $w^+,w^-:\Z\to\W$ such that
\[
    w^-(\hat Z)<h(\hat Z)<w^+(\hat Z),
    \qquad \hat\zeta\text{-a.s.},
    \qquad
    w^\pm\in L^1(\hat\zeta),
\]
and
\[
\begin{aligned}
    &\mathbbm{E}_{\gamma^0}
    \!\left[
        c\bigl(\hat V,\hat W;V^0,w^\pm(\hat Z)\bigr)
    \right]
    <+\infty, \quad 
    &\mathbbm{E}_{\gamma^0}
    \!\left[
        f\bigl(V^0,w^\pm(\hat Z)\bigr)
    \right]
    >-\infty.
\end{aligned}
\]
Then Assumption~\ref{assn:**b} holds.
In particular, suppose $\W$ is an interval and $h(\hat Z)\in\operatorname{int}\W, \; \hat\zeta\text{-a.s.}$. Suppose further that, for some $L<+\infty$ and
$\gamma^0$-almost every realization of
$(\hat V,\hat W,V^0,W^0)$, the mappings
\[
    w\mapsto c(\hat V,\hat W;V^0,w),
    \qquad
    w\mapsto f(V^0,w)
\]
are finite-valued and $L$-Lipschitz on $\W$.
Then \cref{assn:**b} holds.
\end{lemma}
\begin{proof}[Proof of \cref{lem:**a_implies_**b}]
For either sign, define $\gamma^\pm\in \breve\Gamma_{\hat\mu}$
as the law of
    $\bigl(
        \hat V,\hat W,V^0,w^\pm(\hat Z)
    \bigr)$
when $(\hat V,\hat W,V^0,W^0)\sim\gamma^0$. The assumed
integrability (of $w^\pm$ as well as $\gamma^\pm$ against $c$) ensures that
$\gamma^\pm\in\breve\Gamma_{\hat\mu}$.
Moreover, $(A{\gamma^+})(\hat Z)
    =w^+(\hat Z)-h(\hat Z)>0$ and $
    (A{\gamma^-})(\hat Z)
    =w^-(\hat Z)-h(\hat Z)<0$, $\hat\zeta\text{-a.s.}$
By hypothesis, the costs of $\gamma^\pm$ are finite and $\mathbbm{E}_{\gamma^\pm}[f_-]<+\infty.$ The couplings $\gamma^\pm$ therefore
verify \cref{assn:**b}.

For the final assertion, choose a measurable representative
of $h$ taking values in $\operatorname{int}\W$ everywhere,
modifying it on a $\hat\zeta$-null set if necessary.
Set
\[
    \beta(z)
    \coloneqq
    \min\left\{
        1,\,
        \frac12\operatorname{dist}
        \bigl(h(z),\R\setminus\W\bigr)
    \right\},
\]
where $\operatorname{dist}(r,\varnothing)=+\infty$.
Then $\beta$ is measurable, $0<\beta\le1$, and $w^\pm\coloneqq h\pm\beta$
take values in $\W$, satisfy $w^- < h < w^+$, and belong
to $L^1(\hat\zeta)$. Next let
\[
    R^\pm\coloneqq |w^\pm(\hat Z)-W^0| \leq |h(\hat Z)|+1+|W^0|
\]
so that $\mathbb{E}_{\gamma^0}[R^\pm]<+\infty$.
The Lipschitz bounds yield, $\gamma^0$-a.s.,
\[
\begin{aligned}
    c\bigl(\hat V,\hat W;V^0,w^\pm(\hat Z)\bigr)
    \leq
    c(\hat V,\hat W;V^0,W^0)+LR^\pm, \;\; 
    f_-\bigl(V^0,w^\pm(\hat Z)\bigr)
    \le
    f_-(V^0,W^0)+LR^\pm.
\end{aligned}
\]
Taking expectations and using \cref{assn:**a} verifies
the two integrability requirements in the first part
of the lemma, completing the proof.
\end{proof}

\subsection{On \texorpdfstring{\Cref{ass:local-conditional-ui}}{Assumption \ref*{ass:local-conditional-ui}}} \label{sec:on_assn_3}
To help our repair mechanism be successful, we ask that the corrections (made possible by \cref{assn:**a,assn:**b}) to nearly feasible couplings will lead to degradations in objective value that are vanishing. This is what \cref{ass:local-conditional-ui} affords. 
More precisely, \Cref{ass:local-conditional-ui} is a positive-loss tail-control condition, which holds automatically when the objective $f$ is bounded above. For example, when $\mathcal{U}$ is compact and $f$ is upper semicontinuous, as assumed in \cite{jiajin2025unifying}, then \Cref{ass:local-conditional-ui} holds.
When $f$ is possibly unbounded above, the following result demonstrates that \Cref{ass:local-conditional-ui} can be verified through familiar growth conditions relating the objective $f$ to the transportation cost $c$.
\begin{lemma}[Cost-dominated growth]
\label{lem:subcost-lcui}
If for all $\eta>0$, there exists
$a_\eta\in L^1_+(\hat\mu)$ such that, for
$\hat\mu$-a.e.~$\hat u\in\U$,
\begin{equation*} \label{eq: CostDominationCondition}
f_+(u)\leq a_\eta(\hat u)+\eta c(\hat u;u)
\qquad
\text{for every $u\in\U$ with $c(\hat u;u)<+\infty$,}
\end{equation*}
then \Cref{ass:local-conditional-ui} holds. 
\end{lemma}
\begin{proof}
Since $c:\U \times \U \to (-\infty, +\infty]$ is assumed bounded below, it follows that we may take $b \geq 0 $ sufficiently large such that $c+b \geq 0$ over all of $\U\times \U.$ 

Let $\delta>0$ and observe that for all $\gamma\in\mathcal N_\delta$, $E\in\mathcal H$,
\[
\medint_E f_\gamma^+\,d\hat\zeta
=
\mathbbm E_\gamma[f_+(V,W);\hat Z\in E]
\leq
\mathbbm E_{\hat\mu}[a_\eta(\hat U);\hat Z\in E]
+\eta(\rho + b +\delta).
\]
Uniform $L^1$-boundedness follows by setting $E = \Z$ and $\eta = 1$.
Towards confirming uniform integrability, let $\epsilon>0$. Then choose $\eta>0$ so that the second summand
$\eta(\rho + |\inf c| +\delta)<\epsilon/2$. As for the first summand, since $a_\eta \in L_+^1(\hat\mu),$ then
$\mathbbm E_{\hat\mu}[a_\eta(\hat U);\hat Z\in E] < \epsilon/2$ for all $E$ such that $\hat\mu\bigl(H(\hat V) \in E\bigr) = \hat \zeta(E)$ is sufficiently small.  
\end{proof}

When $\U$ lies in a normed vector space and
$c(\hat u;u)\geq\kappa\norm{u-\hat u}^p$ for some $\kappa,p>0$, then \Cref{lem:subcost-lcui} accommodates any loss $f$ in which 
$f_+(u)\leq g(\hat u)+\omega(\norm{u-\hat u}),$ for some $g\in L^1_+(\hat\mu)$ and some $\omega:[0,+\infty)\to\R$ locally bounded above with $\omega(r)=o(r^p)$ as $r \rightarrow \infty$; indeed, then we may take $a_\eta
=
g+
\max\left\{
0,\,
\sup_{r\geq0}
\{\omega(r)-\eta\kappa r^p\}
\right\}.$ In particular, \Cref{lem:subcost-lcui} then accommodates:

\begin{enumerate}[label=\textnormal{(\alph*)},itemsep=2pt,topsep=2pt]
    \item Polynomial growth losses:
    $f_+(u)\leq K(1+\norm{u}^q)$, $0<q<p$, and
    $\mathbbm E_{\hat\mu}[\norm{\hat U}^q]<\infty$;\\
    To verify this, see that we can take
    $g(\hat u)
    =
    K\bigl(1+2^{\max\{q-1,0\}}\norm{\hat u}^q\bigr)$ and $\omega(r)=K2^{\max\{q-1,0\}}r^q$.
    \item Globally Lipschitz losses: $p>1$, $f$ is globally
    $L$-Lipschitz, and $\mathbbm E_{\hat\mu}[f_+(\hat U)]<\infty$.
    \\
    To verify this, see that we can take $g(\hat u)=f_+(\hat u)$ and $\omega(r)=Lr$.
\end{enumerate}

\subsection{Repair-based Normality}
\label{sec:main_result}
We now present our normality result, which leverages a repair mechanism that we detail in its proof. 

\begin{restatable}{theorem}{MainResult}
    \label{prop:main_result}
    In the above setting, if Assumptions \ref{assn:**a}, \ref{assn:**b} and \ref{ass:local-conditional-ui} hold,
    then
    $
    \eqref{eq:primal}
    = \dualPsiRef,
    $
    with $\Psi= L^\infty(\hat \zeta).$
\end{restatable}
\begin{proof}[Proof of \cref{prop:main_result}]
Set
\(
a\coloneqq \rho-\mathbbm E_{\gamma^0}[c]>0,
\)
and to recover the dual \eqref{eq:dual_Psi}, define
\[
\Fdot(\tau,\theta;\gamma)
\coloneqq
\mathbbm E_\gamma[f(V,W)]
-\iota_{(-\infty,\tau]}
    \bigl(\mathbbm E_\gamma[c]-\rho\bigr)
-\iota_{\{\theta\}}(A\gamma)
-\iota_{\breve\Gamma_{\hat\mu}}(\gamma)
\]
over $\R\times L^1(\hat\zeta)\times
\mathcal M(\U\times\U)$, and let
\[
\pdot(\tau,\theta)
\coloneqq \sup_\gamma \Fdot(\tau,\theta;\gamma).
\]
Then $\pdot(0,0)=\eqref{eq:primal_explicit}$, which equals \eqref{eq:primal} by \cref{lem:v1=v2}. Moreover,
\[
\Fdot_{\sf d}(0;\tau^*,\theta^*)
=
\begin{cases}
\displaystyle
{\textstyle\sup_{\gamma\in\breve\Gamma_{\hat\mu}}}
\left\{
\mathbbm E_\gamma[f]
+\tau^*\bigl(\rho-\mathbbm E_\gamma[c]\bigr)
-\mathbbm E_{\hat \zeta}[\theta^* \, A\gamma]
\right\},
& \tau^*\ge0,\\[1.2ex]
+\infty,&\tau^*<0.
\end{cases}
\]
Consequently, identifying $\lambda=\tau^*$ and $\psi=\theta^*$, we obtain
\[
\qdot(0)
\coloneqq \inf_{\tau^*,\theta^*}
\Fdot_{\sf d}(0;\tau^*,\theta^*)
=(\dualLinfty).
\]
It suffices to show that $\pdot(0,0) = \qdot(0)$

So that we may leverage \cref{prop:convex_duality}, we proceed to show that $\pdot$ is proper and $\pdot(0,0)\in\R.$ By \cref{assn:**a}, $\gamma^0\in\mathcal N_\delta$ is feasible for
$\pdot(0,0)$ and satisfies
$\mathbbm E_{\gamma^0}[f_-]<+\infty$, so
$\pdot(0,0) > - \infty$. By
\Cref{ass:local-conditional-ui}, for some $\delta>0$,
\(
T_\delta
\coloneqq
\sup_{\gamma\in\mathcal N_\delta}
\mathbbm E_\gamma[f_+]<+\infty;
\)
hence $\pdot$ is bounded above by $T_\delta$ on a neighborhood of
$(0,0)$. Altogether, it follows that $\pdot(0,0)\in\R$; furthermore, the concavity and local upper boundedness rule out the possibility that $p$ is
$+\infty$ at any $(\tau, \theta)$,  and therefore $p$ is proper. By \Cref{prop:convex_duality}\ref{convex_duality:normality}, it suffices to show that 
\[
\limsup_{(\tau,\theta)\to(0,0)}\pdot(\tau,\theta) \leq \pdot(0,0) + \epsilon \quad \quad \forall \epsilon > 0. 
\]
Let $\epsilon > 0.$
If $\limsup_{(\tau,\theta)\to(0,0)}\pdot(\tau,\theta)=-\infty$, there is nothing to prove.
Otherwise, let $\{(\tau_r, \theta_r)\}_{r=1}^\infty$ be a sequence such that  $(\tau_r,\theta_r)\to(0,0)$ strongly, $\pdot(\tau_r,\theta_r)\to \limsup_{(\tau,\theta)\to(0,0)}\pdot(\tau,\theta)$
and $\pdot(\tau_r,\theta_r)>-\infty$; for each $r$ we can also let $\gamma^r$ be a feasible decision for $\pdot(\tau_r, \theta_r)$ satisfying 
\[
\mathbbm E_{\gamma^r}[f]
\geq
\pdot(\tau_r,\theta_r)-\frac1r.
\] 
We proceed to repair (over two edits) the infeasibility of $\gamma^r$ for the problem $\pdot(0,0)$ and then show the asymptotic optimality of this repaired measure.

\paragraph{Feasibility Repair}
{\bf \underline{Transport-cost edit}:}
For all sufficiently large $r$, $\gamma^r\in\mathcal N_\delta$.
Set
\[
\kappa_\epsilon
\coloneqq
\min\left\{
\frac12,
\frac{\epsilon}
{T_\delta-\mathbbm E_{\gamma^0}[f]+1}
\right\}.
\]
Define the transport-cost-repaired coupling $\bar \gamma_r$ by
\[
\bar\gamma^r
\coloneqq
(1-t_r)\gamma^r+t_r\gamma^0,
\quad
t_r
\coloneqq
\max\left\{0 ,\frac{\tau_r+a\kappa_\epsilon}{\tau_r+a} \right\} \in[0,1]
.
\]
Since $A\gamma^r=\theta_r$ and $A\gamma^0=0$,
\[
\mathbbm E_{\bar\gamma^r}[c]
\le \rho-a\kappa_\epsilon,
\qquad
\|A\bar\gamma^r\|_{L^1(\hat\zeta)}
\le
\frac{a-a\kappa_\epsilon}{\tau_r+a}
\|\theta_r\|_{L^1(\hat\zeta)}
=o(1).
\]
Thus $\bar\gamma^r\in\mathcal N_\delta$ for all sufficiently large $r$.
Moreover,
\begin{align}
\label{eq:sequence_obj_gap}
\mathbbm E_{\gamma^r}[f]-\mathbbm E_{\bar\gamma^r}[f]
=
 t_r\bigl(
 \mathbbm E_{\gamma^r}[f]-\mathbbm E_{\gamma^0}[f]
 \bigr)
\le
 t_r\bigl(T_\delta-\mathbbm E_{\gamma^0}[f]\bigr)
\le \epsilon+o(1).
\end{align}

\noindent{\bf \underline{Conditional-moment edit}:}
We now mix $\bar\gamma^r$ with $\gamma^+$ and $\gamma^-$ to generate a measure $\tilde\gamma^r$ that is feasible for the problem $\pdot(0,0)$ when $r$ is sufficiently large. 
For all $\hat z$ in a set of  $\hat\zeta$-measure 1, define the nonnegative functions
\[
t_r^+(\hat z)
\coloneqq
\frac{
\max\{0, -A\bar\gamma^r(\hat z)\}
}
{|A\bar\gamma^r(\hat z)|+|A\gamma^+(\hat z)|},
\qquad
 t_r^-(\hat z)
\coloneqq
\frac{
\max\{0, A\bar\gamma^r(\hat z)\}
}
{|A\bar\gamma^r(\hat z)|+|A\gamma^-(\hat z)|},
\]
and set them equal to zero everywhere else.
Crucially, for $\hat\zeta$-a.s. $\hat z$, the mixture of residuals is zero, i.e.,
\[
(1 - (t_r^+ + t_r^-)) A\bar \gamma^r
        + t_r^+ A\gamma^+
        + t_r^- A\gamma^- = 0.
\]
So using these weights, we define a conditional-moment-repaired coupling $\tilde\gamma^r$ by
\[
\begin{aligned}
\tilde\gamma^r(B)
\coloneqq
\mathbbm E_{\bar \gamma^r}\Bigl[
    \bigl(1-(t_r^+ + t_r^-)(\hat Z)\bigr)\ind_B
\Bigr]
+\mathbbm E_{\gamma^+}\bigl[ t_r^+(\hat Z) \ind_B \bigr]
+\mathbbm E_{\gamma^-}\bigl[ t_r^-(\hat Z) \ind_B \bigr],
\quad
B\in\mathcal F\otimes\mathcal F.
\end{aligned}
\]

\noindent{\bf \underline{Verifying edit-feasiblity}:}
For starters, $\tilde\gamma^r \in \mathcal{P}(\U \times \U)$; indeed, its nonnegativity is apparent, and 
because $\bar\gamma^r,\gamma^+,\gamma^-$ all have first marginal
$\hat\mu$, they also have the same $\hat Z$-law $\hat\zeta$, so
\[
\tilde\gamma^r(\U\times\U)
=
\mathbbm E_{\hat \zeta}[1-(t_r^+ + t_r^-)] + \mathbbm E_{\hat \zeta}[t_r^+] + \mathbbm E_{\hat \zeta}[t_r^-]
=1.
\]
Next, $\projection_{\hat U}\tilde\gamma^r=\hat\mu$ follows from the fact that since $\hat Z = H(\hat V)$, for every $D\in\mathcal F$ we have
\begin{align*}
\tilde\gamma^r(D\times\U)
&=
\mathbbm E_{\hat \mu}\bigl[ \ind_D(\hat U) (1- (t_r^+ + t_r^-))(H(\hat V))
+ \ind_D(\hat U) t_r^+(H(\hat V))
+\ind_D(\hat U) t_r^-(H(\hat V))\bigr]\\
&=\mathbbm{E}_{\hat\mu}\bigl[ \ind_D(\hat U)\bigr] = \hat\mu(D).
\end{align*}
As well, $\tilde\gamma^r\in\breve\Gamma$,
since 
\(
\mathbbm E_{\tilde\gamma^r}|W|
\le
\mathbbm E_{\bar\gamma^r}|W|
+\mathbbm E_{\gamma^+}|W|
+\mathbbm E_{\gamma^-}|W|<+\infty
\)
and
$
\mathbbm E_{\tilde\gamma^r}[c_+]
\leq
\mathbbm E_{\bar\gamma^r}[c_+]
\mathbbm E_{\gamma^+}[c_+]
\mathbbm E_{\gamma^-}[c_+] < +\infty
.
$
As for the conditional moment constraint,  since for every bounded
measurable $\varphi:\Z\to\R$, the definition of conditional
expectation and the definition of $\tilde\gamma^r$ give
\begin{align*}
&\mathbbm E_{\tilde\gamma^r}
\bigl[\varphi(\hat Z)(W-h(\hat Z))\bigr] = \mathbbm{E}\left[\varphi(\hat Z) \cdot \mathbbm E_{\tilde\gamma^r}
\bigl[(W-h(\hat Z)) \mid \hat Z\bigr]\right]
\\
&\quad=
\mathbbm E_{\hat \zeta} \Bigl[
    \varphi \cdot \bigl(
        (1 - (t_r^+ + t_r^-)(\hat Z)) A\bar \gamma^r(\hat Z)
        + t_r^+(\hat Z) A\gamma^+(\hat Z)
        + t_r^-(\hat Z) A\gamma^-(\hat Z)
    \bigr)
\Bigr]
=0,
\end{align*}
and it follows that
\(
A\tilde\gamma^r=0,\;\hat\zeta\text{-a.e.}
\)

What remains is to show that $\mathbbm E_{\tilde\gamma^r}[c] \leq \rho$ for sufficiently large $r.$ Towards this, it will be useful to note that the amount of mixing weight, $t_r^+ + t_r^-$, that is given to $\gamma^+$ or $\gamma^-$, vanishes.
Indeed, with
\[
b(\hat z)
\coloneqq
\min\{A\gamma^+(\hat z),-A\gamma^-(\hat z)\}>0
\quad \hat\zeta\text{-a.e.},
\]
for every $\eta>0$,
\[
\mathbbm E_{\hat \zeta}[(t_r^+ + t_r^-)]
\leq 
\mathbbm E_{\hat \zeta} \left[  \frac{ \abs{A\bar\gamma^r} }{\abs{A\bar\gamma^r} + b}
\right]
\leq 
\hat\zeta\{b\le\eta\}
+\frac1\eta\|A\bar\gamma^r\|_{L^1(\hat\zeta)}.
\]
Taking $r\to\infty$ and then $\eta\downarrow0$ gives
\begin{equation}
\label{eq:tmix_vanish}
\|t_r^+ + t_r^-\|_{L^1(\hat\zeta)}\to0.
\end{equation}
We will use the elementary consequence of uniform
integrability that if $t_r^+ + t_r^-\in[0,1]$, $\|t_r^+ + t_r^-\|_1\to0$, and $\{g_r\}\subset L^1_+(\hat\zeta)$ is uniformly integrable, then
\begin{equation}
\label{eq:tmix_ui_vanish}
\mathbbm E_{\hat \zeta} [(t_r^+ + t_r^-) g_r]\to0.
\end{equation}
Indeed, for every $M>0$,
\[
\mathbbm E_{\hat\zeta}[ (t_r^+ + t_r^-) g_r]
\leq
M\|t_r^+ + t_r^-\|_{L^1(\hat\zeta)}
+
\mathbbm E_{\hat\zeta}[g_r; g_r>M],
\]
and then one lets first $r\to\infty$ and then $M\to\infty$. As $\inf c > -\infty$, we can without loss of generality replace $(c,\rho)$ by
$(c-\inf c,\rho-\inf c)$ if necessary, so that $c\ge0$. 
Let $c^\pm\in L^1_+(\hat\zeta)$
be measurable versions satisfying
\[
c^\pm(\hat Z)
=
\mathbbm E_{\gamma^\pm}[c\mid \hat Z],
\qquad \gamma^\pm\text{-a.s.}
\]
Since $c_+ + c_-\in
L^1_+(\hat\zeta)$, \eqref{eq:tmix_vanish} and \eqref{eq:tmix_ui_vanish} imply
\begin{align*}
\mathbbm E_{\tilde\gamma^r}[c]
&\leq
\mathbbm E_{\bar\gamma^r}[c]
+\mathbbm E_{\hat \zeta} [t_r^+ c^+]
+\mathbbm E_{\hat \zeta} [t_r^- c^-]
\\
&\leq
\rho-a\kappa_\epsilon
+\mathbbm E_{\hat\zeta} [(t_r^+ + t_r^-)(c^+ + c^-)]
\\
&=
\rho-a\kappa_\epsilon+o(1)
\le \rho
\end{align*}
for all sufficiently large $r$.  Thus $\tilde\gamma^r$ is eventually
feasible for $\pdot(0,0)$.

\paragraph{Asymptotically Negligible Deterioration in Objective by the Repair}
It remains to control the objective loss from the second edit.
Let $f_{\gamma^+}^-, f_{\gamma^-}^-$ be measurable functions satisfying
\[
f_{\gamma^+}^-(\hat Z) = \mathbbm E_{\gamma^+}[f_-\mid \hat Z],\quad\text{and}\quad
f_{\gamma^-}^-(\hat Z) = \mathbbm E_{\gamma^-}[f_-\mid \hat Z].
\]
Since $\bar\gamma^r\in \mathcal{N}_\delta$ for large enough $r$, $\{f^+_{\bar \gamma^r}\}_r$ is uniformly integrable by \cref{ass:local-conditional-ui} while $f_{\gamma^+}^-, f_{\gamma^-}^-\in L^1(\hat \zeta)$ by \cref{assn:**b}.
From the definition of $\tilde\gamma^r$, dropping the nonpositive terms yields
\begin{align}
\label{eq:edit_obj_gap}
\mathbbm E_{\bar\gamma^r}[f]
-\mathbbm E_{\tilde\gamma^r}[f]
&\leq
\mathbbm E_{\hat\zeta} [(t_r^+ + t_r^-) \, f^+_{\bar\gamma^r}]
+\mathbbm E_{\hat\zeta} [t_r^+ f^-_{\gamma^+}]
+\mathbbm E_{\hat\zeta} [t_r^- f^-_{\gamma^-}]
=o(1),
\end{align}
where the last equality follows from \eqref{eq:tmix_vanish} and \eqref{eq:tmix_ui_vanish} applied to the uniformly integrable family $\{f^+_{\bar \gamma^r}\}_r$ as well as to $f^-_{\gamma^+},f^-_{\gamma^-}$, together with noting that $\max\{t_r^+, t_r^-\} \leq t_r^+ + t_r^-$.
Combining \eqref{eq:sequence_obj_gap} and \eqref{eq:edit_obj_gap} gives
\[
\mathbbm E_{\gamma^r}[f]
-\mathbbm E_{\tilde\gamma^r}[f]
\leq \epsilon+o(1).
\]
Since $\tilde\gamma^r$ is eventually feasible for $\pdot(0,0)$,
\[
\pdot(\tau_r,\theta_r)-\frac1r
\leq
\mathbbm E_{\gamma^r}[f]
\leq
\pdot(0,0)+\epsilon+o(1).
\]
Letting $r\to\infty$ and then $\epsilon\downarrow0$ gives $\limsup_{(\tau,\theta)\to(0,0)}\pdot(\tau,\theta)\leq\pdot(0,0)$.
Therefore we find
\(
\eqref{eq:primal}
=\pdot(0,0)
=\qdot(0)
=\dualPsiRef
\)
and conclude.
\end{proof}

We now comment on the nature of this duality under our assumptions.
First our assumptions are distinct from the Slater condition \cref{prop:convex_duality}\ref{cond:slater} (see \cref{ex:not_slater}), differentiating our results from the traditional theory.
Second the existence of either primal or dual solutions is not guaranteed under \cref{assn:**a,assn:**b} (see \cref{ex:nonexistence_primal,ex:nonexistence_dual,ex:nonexistence_dual_Cb} and \cref{lem:dual-nonexist-template} in \PaperSupplementSection{apx:nonexistence}).
This indicates that our work is also distinct from results obtained by the classic Fenchel Duality Theorem—see \cite{Rockafellar.70} for the original, finite-dimensional statement, \cite{brezis2011functional} for an infinite-dimensional statement, and \cite{vancuong2023generalized} for a more recent generalization—or
for that matter, any result that guarantees the existence of a primal or dual solution.

That said, there are a few specializations of our setting that warrant additional discussion.
For instance, if $\hat\zeta$ is finitely supported, (e.g. $\Z$ is finite, $H(\V)$ is finite, or $\hat\nu$ is finite-supported),
then $L^1(\hat\zeta)$ is finite dimensional,
and the classical finite dimensional range interiority condition (\cref{prop:convex_duality}\ref{cond:slater}), gives the following.
\begin{proposition}[Finitely supported $\hat \zeta$]
    \label{prop:finite-H-support}
    Suppose \cref{assn:**a,assn:**b} hold, $\hat\zeta$ is finitely supported, and $\eqref{eq:primal} < + \infty.$
    Then $\eqref{eq:primal}=\dualPsiRef$ 
    with dual attainment, 
    for $\Psi = L^\infty(\hat\zeta) \cong \reals^m$, where $m$ is the size of the support of $\hat\zeta$.
\end{proposition}
\noindent Several canonical DRO models in fact have corresponding CM-OT instances with finitely supported $\hat\zeta$. This includes both Wasserstein DRO \citep{gao2022short} and $\phi$- divergence DRO \citep{shapiro2017dro}. Sinkhorn DRO \citep{wang2025sinkhorn} also exhibits finitely supported $\hat \zeta$ when the first marginal to the reference distribution is finitely supported. We refer the reader to \cite{jiajin2025unifying} and \cite{chen2026perturbation} for the details on their reductions to CM-OT. 
As for \cref{assn:**a,assn:**b}, these hold under the respective classical assumptions of the models. Thus, \cref{prop:finite-H-support} suggests that the value of \cref{ass:local-conditional-ui} is really in helping to achieve normality in the infinite dimensional range case.

Finally, we remark that the choice of $\Psi=L^\infty(\hat\zeta)$ in \cref{prop:main_result} is motivated by the perturbation framework, since $A\gamma$ residuals naturally lie in $L^1(\hat\zeta)$ with $L^\infty(\hat\zeta)$ multipliers.
We next ask when these multipliers may be replaced with the more computationally friendly $\mathcal C_b(\Z)$.

\section{Multiplier Spaces and Interchangeability  \texorpdfstring{$[(\dualLinftySym)=(\dualCbSym+\ip)]$}{DLinfty = DCb+IP}}

\label{sec:InterchangeabilityPrinciple}
By \cref{prop:main_result}, $\eqref{eq:primal}=\dualPsiRef$ holds when $\Psi=L^\infty(\hat\zeta)$.
To complete CM-OT duality, it remains to establish the separate, measure-theoretic equality 
$(\dualLinfty) =
(\dualCbIP)$.
We do so in two steps.
First, in \cref{prop:ip}, we appeal to classical measurable-selection theory \cite{bertsekasshreve1978,brownpurves1973} to show how we may obtain $\dualPsiRef=\dualPsiIPRef$ for any $\Psi\subseteq L^\infty(\hat \zeta)$.
Second, in \cref{thm:Cb_Linfty}, under additional tail control conditions,  we show how the multiplier space can be restricted from $L^\infty(\hat\zeta)$ to $\mathcal{C}_b(\Z)$ while preserving the dual value to obtain $(\dualLinfty) = (\dualCb)$.

In this section, it will be useful notationally to write the dual objective as
\begin{align*}
\mathcal J(\lambda,\psi) &\coloneqq
\lambda\rho +
{\textstyle\sup_{\gamma\in\breve\Gamma_{\hat\mu}}}\;\;
\mathbbm{E}_{\gamma} [J_{\lambda,\psi}(\hat U,U)] \\
&\leq \mathcal{J}^{IP}(\lambda, \psi) \coloneqq \lambda\rho + \mathbbm{E}_{\hat\mu}\left[\sup_{u \in \U} J_{\lambda, \psi}(\hat U, u)\right],
\end{align*}
where
\(J_{\lambda,\psi}(\hat u, u) \coloneqq f(u) - \lambda c(\hat u;u) - (\psi \circ H)(\hat v) \cdot \left[w-(h\circ H)(\hat v)\right].
\)
Indeed, the following result establishes the first step, \(
\dualPsiRef=\dualPsiIPRef,
\) by seeking out when $\mathcal J$ and $\mathcal{J}^{IP}$ are (nearly) equivalent.  
As measurable selections will be leveraged, all universally measurable functions and sets arising from measurable selection are understood with respect to the completions of the relevant probability measures.

\begin{restatable}[(Classical) Interchangeability Principle]
{proposition}{InterchangeabilityPrinciple}
\label{prop:ip}
In the setting above
suppose additionally that $(\V,\mathcal G)$ is a standard Borel space, and suppose that there exists a $\breve\gamma\in\breve\Gamma_{\hat\mu}$
satisfying $\mathbbm E_{\breve\gamma}[\abs{f}] < +\infty$ and $\mathbbm E_{\breve\gamma}[c]<+\infty$.
Then for every $\Psi\subseteq L^\infty(\hat\zeta)$, the interchange
\(
\dualPsiRef=\dualPsiIPRef
\)
holds.
\end{restatable}

Since $c$ is bounded below and integrable under every $\gamma\in\breve\Gamma_{\hat\mu}$, we have
\(
\mathcal J(\lambda+\epsilon,\psi)\leq\mathcal J(\lambda,\psi)+\epsilon(\rho-\inf c)
\)
for every $\epsilon>0$, $\lambda\geq0$, and $\psi\in\Psi$.
Consequently,
\begin{align}
    \label{eq:positive_lambda_identity}
    \eqref{eq:dual_Psi} \equiv
    \textstyle \inf_{\lambda\geq 0,\psi\in\Psi}\mathcal J(\lambda,\psi)
    =
    \inf_{\lambda>0,\psi\in\Psi}\mathcal J(\lambda,\psi),
\end{align}
so it therefore suffices to establish $\mathcal J^{\ip}(\lambda,\psi) = \mathcal J(\lambda,\psi)$ for all $(\lambda, \psi) \in \reals_{>0} \times \Psi$, as then
\begin{align}
    \dualPsiIPRef
    \leq\inf_{\substack{\lambda> 0,\\\psi\in\Psi}}\mathcal J^{IP}(\lambda,\psi)
    = \inf_{\substack{\lambda> 0,\\\psi\in\Psi}}\mathcal J(\lambda,\psi)
    \overset{\eqref{eq:positive_lambda_identity}}{=} \inf_{\substack{\lambda\geq 0,\\\psi\in\Psi}}\mathcal J(\lambda,\psi)
    \equiv \dualPsiRef,
    \label{eq:ip-finish}
\end{align}
as desired.
More precisely, the proof of \Cref{prop:ip} shows that, for each
$(\lambda,\psi)\in\mathbb R_{>0}\times\Psi$,
the value
$\mathbbm E_{\hat\mu}
\left[
\sup_{u\in\U}J_{\lambda,\psi}(\hat U,u)
\right]$
can be approximated arbitrarily closely from below by
$\mathbbm E_\gamma
\left[
J_{\lambda,\psi}(\hat U,U)
\right]$
for suitably constructed
$\gamma\in\breve\Gamma_{\hat\mu}$.
The construction combines measurable selection with the reference
coupling $\breve\gamma$ in order to preserve the required
integrability conditions. Details are deferred to
Appendix~\ref{pf:prop:ip}.

Applying \cref{prop:ip} with $\Psi=\mathcal C_b(\Z)$ gives $(\dualCb)=(\dualCbIP)$.
It remains to identify conditions such that $(\dualLinfty)=(\dualCb)$.
The following result presents sufficient conditions for this equivalence of multiplier spaces. 

\begin{theorem}[From $L^\infty(\hat\zeta)$ to $\mathcal{C}_b(\Z)$]
\label{thm:Cb_Linfty}
Suppose $\Z$ is a Polish space and at least one of the following holds:
\begin{enumerate}[label=\textnormal{(\roman*)},itemsep=2pt,topsep=2pt]
    \item
    The family $\bigl\{ \mathbbm E_\gamma[\abs{W}\mid\hat Z=\cdot]: \gamma\in\breve\Gamma_{\hat\mu} \bigr\} \subseteq L^1_+(\hat\zeta)$  is uniformly integrable;

    \item
    For every $\eta>0$, there exists $a_\eta\in L^1_+(\hat\mu)$ such that $\abs{w} \leq a_\eta(\hat u) + \eta c(\hat u;u)$ for $\hat\mu$-a.e.~$\hat u$ and every $u\in\U$ with $c(\hat u;u)<+\infty$.
\end{enumerate}
Then
\(
(\dualLinfty)=(\dualCb).
\)
\end{theorem}

We defer the proof of \cref{thm:Cb_Linfty} to Appendix~\ref{pf:prop:ip}, in which the strategy is to show that for each finite-value pair $(\lambda,\psi)$ and any $\epsilon>0$, we can construct $(\lambda',\psi')\in\R_+\times\mathcal C_b(\Z)$ such that $\mathcal{J}(\lambda',\psi') \leq \mathcal{J}(\lambda, \psi) + \epsilon$.
This follows because any $\psi \in L^\infty(\hat\zeta)$ agrees with a $\psi' \in \mathcal C_b(\Z)$ outside a set of arbitrarily small measure (by Lusin's and Tietze Extension). The disagreement between $\psi$ and $\psi'$ over this small-measure set leads to a small difference in dual objective $\mathcal{J}$ when the expectation of W over the disagreement set can be controlled. Condition (i)'s uniform integrability is a direct approach; condition (ii) ensures that large $W$-values on the disagreement set must either be controlled by an integrable term or incur correspondingly large transport cost, in which the latter can be controlled when $\lambda'$ approaches $\lambda$ from above.

Combining our previous results gives:\footnote{When $\eqref{eq:primal} < +\infty$, we note that $\gamma^0$ from \cref{assn:**a} provides a $\breve\gamma$ in \cref{prop:ip}; indeed, $\mathbbm{E}_{\gamma^0}[c] < \rho < +\infty$ and feasibility of $\gamma^0$ gives $\mathbbm E_{\gamma^0}[f]\leq\eqref{eq:primal}<+\infty$, which combined with $\mathbbm E_{\gamma^0}[f_-]<+\infty$ reveals $\mathbbm E_{\gamma^0}[\abs{f}]<+\infty$. If $\eqref{eq:primal} = +\infty,$ the conclusion of \cref{thm:combined} holds trivially (e.g. weak duality).  
}
\begin{theorem}
    \label{thm:combined}
    In the above setting, suppose 
    Assumptions \ref{assn:**a}, \ref{assn:**b} and \ref{ass:local-conditional-ui} hold;     %
    $(\V,\mathcal G)$ is standard Borel, $\Z$ is Polish,
    and \textnormal{[(i)} $\lor$ \textnormal{(ii)]} of \cref{thm:Cb_Linfty} holds.
    Then,
    \[
    \eqref{eq:primal}
    =
    (\dualLinfty)
    =
    (\dualCb)
    =
    (\dualCbIP).
    \]
\end{theorem}

\section{Observation-Calibrated Divergence DRO}\label{sec:: Applications}
In this section, we
present and examine a general model that lies at the intersection of divergence-based risk analysis, distributional calibration, conditional risk measures, and stochastic inverse problems. It is itself a special case of CM-OT, making its duality analysis amenable to our repair-based conditions. To formulate it, let us recall the setting of \cref{sec:cmot-problem-setting} but now assume that $(\V,\mathcal G)$ is standard Borel and restrict the target function $h\in L^1(\hat\zeta)$ to be a density, i.e,
$h \in L^1_+(\hat\zeta)$  and $\mathbbm{E}_{\hat Z \sim \hat\zeta}[h(\hat Z)] = 1$. Additionally, let
$\phi: \reals \to \Reals_+$ be a proper, convex function such that $\phi(t)=+\infty$ for $t<0$, $\phi(1) = 0$ and $\phi(0) = \lim_{t\downarrow 0} \phi(t)$. As well, let $\ell:\V \rightarrow \reals$ be a measurable loss function.  

\subsection{Observation-Calibrated Ambiguity}
We interpret $\hat V \sim \hat\nu$ as a reference model for an uncertain state and $\hat Z = H(\hat V) \sim H_\#\hat\nu$ as the corresponding observation of that state under the measurable information map
$H:\V\to\mathcal Z.$ Since $h$ is now presumed to be a density, then $h\hat\zeta$ specifies a \emph{target observation distribution} with which we may enforce a calibration of state models. More precisely, we may consider alternative state models $V\sim\nu$ that satisfy two requirements: the state distribution $\nu$ does not depart excessively from the reference model $\hat\nu$, as measured by a $\phi$-divergence, and the observation 
$Z=H(V)\sim H_\#\nu$ made about $V$ agrees exactly with the prescribed target distribution
$h\hat\zeta$. The latter requirement will be referred to as the \emph{calibration constraint}. This gives the following \emph{observation-calibrated $\phi$-divergence DRO problem}:
\begin{align}
\label{eq:conditional_model}
R_{\mathsf{H}}(\ell)\coloneqq
\left\{
\begin{array}{cl}
     \displaystyle \sup_{\nu\in\mathcal{P}(\V)} & \mathbbm{E}_{\nu}[\ell] \\[1mm]
     \st
     & \mathbbm{D}_{\phi}(\hat\nu,\nu) \leq \rho\\
     & H_\# \nu = h\,\hat\zeta
\end{array}
\right\},
\quad \mathsf{H} \coloneqq (H,h),\;\;
\tag{$\mathrm{P}_\phi\mid \mathsf{H}$}
\end{align}
where $\mathbbm{D}_\phi(\hat\nu,\nu)\coloneqq \mathbbm{E}_{\hat V \sim \hat \nu}\bigl[\phi\bigl(\frac{d\nu}{d\hat\nu}(\hat V)\bigr)\bigr]$  when $\nu\ll\hat\nu$ and $+\infty$ otherwise, and $0\leq\rho<+\infty$ is the budget.

When $h\equiv1$, the calibration constraint requires every admissible state model $\nu$ to induce exactly the same observation distribution as the reference model. i.e., $H_\#\nu=\hat\zeta = H_\#\hat\nu$. More generally, if $\zeta\in\mathcal P(\mathcal Z)$ satisfies
$\zeta\ll\hat\zeta$, then choosing
$h=\frac{d\zeta}{d\hat\zeta}$
calibrates the alternative state model to the prescribed observation distribution
$\zeta$. The observation-calibrated model is most meaningful when the observation map \(H\) is many-to-one. If \(H\) is injective, and in particular if \(H=\mathrm{id}\), then prescribing \(H_\#\nu\) largely determines the state distribution \(\nu\) itself, leaving little or no residual ambiguity. In contrast, when \(H\) is a coarser representation of the state, the constraint \(H_\#\nu=h\,\hat\zeta\) fixes only the distribution of the observable feature \(H(V)\), while allowing the adversary to redistribute probability among states lying in the same fiber \(H^{-1}(z)\). Thus, the model is calibrated with respect to the prescribed observations while remaining distributionally robust with respect to latent variation that those observations do not resolve.

\subsubsection{Related Models and Applications}
The observation-calibrated ambiguity set, or feasible region to $\eqref{eq:conditional_model}$, is related to several strands of existing work. \citet{GlassermanXu.14} study entropy-penalized model uncertainty while preserving the distribution of a specified component of the uncertain state, while \citet{DeyJunejaMurthy.15} consider minimum-divergence calibration of probability models subject to prescribed marginal distributions. A closely related prescribed-pushforward perspective arises in data-consistent stochastic inversion: \citet{ButlerJakemanWildey.18} construct state distributions whose pushforward through a specified observation map agrees with an observed distribution, and \citet{jiang2026iterativedataconsistentinversionmultiple} recently establish \(f\)-divergence optimality properties for such pushforward-consistent distributions. Our formulation differs from these projection and calibration problems by retaining the entire divergence-bounded family of observation-consistent distributions and optimizing worst-case expected loss over that family. Related conditional perspectives arise in the theory of conditional divergence risk measures and conditional distributionally robust functionals \citep{PrincipiMaccheroni.25,ShapiroPichler.24}. Similar distinctions between a prescribed or identifiable context distribution and ambiguity in outcomes conditional on context also arise in robust prediction and learning \citep{AiRen.24,SahooLeiWager}.

We next briefly summarize some motivating applications that build on a common modeling theme. Let $\hat V=(\hat X,\hat Y)$ consist of covariates $\hat X$ and a response $\hat Y$.
A basic target-population prediction model takes $H(X,Y)=X$, so that the target covariate distribution is prescribed while the adversary remains free to perturb the conditional response law $Y\mid X$ subject to the divergence budget.
More generally, $H(X,Y)=s(X)$ may retain only a trusted lower-dimensional feature (such as demographic strata, a risk score, or another summary of the covariates) while allowing the distribution of $X$ itself to vary within the fibers of $s$.
The map $H$ may also incorporate information about the response.
For example, if external data reliably specify the prevalence of an event $\{Y\in B\}$, then $H(X,Y)=\mathbbm1_{\{Y\in B\}}$ constrains this prevalence while allowing the adversary to alter which covariate profiles experience the event.
Finally, joint choices such as
\[
H(X,Y)=\bigl(s(X),\mathbbm1_{\{Y\in B\}}\bigr)
\]
can prescribe both the distribution of selected covariate features and corresponding outcome prevalences, while robustifying the finer distribution within each resulting stratum.
Thus, the same formulation ranges from ordinary target-covariate adjustment to prediction informed by externally available aggregate outcome information.

\paragraph{Contextual inventory and resource allocation.}
Let us consider an inventory setting in which the uncertain state is \(V=(Z,Y)\), where \(Z\) denotes an observed context (for example, market condition, or another covariate) and $Y \in \reals$ denotes demand. Suppose an ordering policy $a: \Z \rightarrow \reals_+$ leads to the order quantity \(a(Z)\) after observing \(Z\). In an inventory setting, the loss is then
\[
\ell_a(z,y)
\coloneqq
b\bigl(y-a(z)\bigr)_+
+
s\bigl(a(z)-y\bigr)_+,
\]
where \(b,s>0\) are shortage and overage costs. $\hat V = (\hat Z, \hat Y) \sim \hat\nu$ and $\hat Z \sim \hat \zeta$ are reference models, possibly derived from historical data. With interest in making ordering decisions under a different set of market conditions from that of the reference $\hat\zeta$, the observation-calibration constraint
prescribes instead $h\hat\zeta$ as the distribution of market conditions, while allowing the conditional distribution of demand \(Y\) given context \(Z\) to vary from that in $\hat \nu$, subject to the divergence budget. The resulting robust inventory decision problem is then
\[
\inf_{a\in\mathcal A}
\sup \{
\mathbbm E_\nu[\ell_a(Z,Y)]: \;\; \nu \in \mathcal{P}(\V), \;\; \mathbbm D_\phi(\hat\nu,\nu)\leq\rho, \;\; H_\#\nu=h\,\hat\zeta \}.
\]
The same construction applies more generally to resource-allocation problems in which decisions are conditioned on observed covariates and performance depends on an uncertain outcome whose conditional law is only approximately known.

\paragraph{Reliability under a prescribed operating environment.}
A second application arises in reliability analysis. Let \(V=(Z,Y)\), where \(Z\) denotes an operating environment---for example, load, temperature, mission profile, or usage intensity---and \(Y\) denotes an uncertain system response, resistance, or degradation state. For a design decision \(x\), let
$\mathcal F_x
\coloneqq
\{(z,y):G(x,z,y)>0\}$
denote the failure region, and set
$\ell_x(z,y)\coloneqq\mathbf 1_{\mathcal F_x}(z,y).$ Under \(H(z,y)=z\), the worst-case failure probability under a prescribed operating-environment distribution is
\[
p_{\rm fail}(x)
\coloneqq \sup \{\nu(\mathcal{F}_x): \;\; \nu \in \mathcal{P}(\V), \;\; \mathbbm D_\phi(\hat\nu,\nu)\leq\rho, \;\; H_\#\nu=h\,\hat\zeta \},
\]
so a robust design problem may therefore take the form
\[
\inf \{C(x): x \in X,  \;\; p_{\rm fail}(x) \leq \alpha \}
\]
where \(C(x)\) is the design cost and \(\alpha\) is a prescribed reliability tolerance.

\subsection{Duality in Conditional Divergence}
Here we analyze \eqref{eq:conditional_model} through its CM-OT representation. Indeed, the following result establishes that it can be represented as a CM-OT problem, so we can leverage our repair-based analysis of CM-OT to address primal-dual equality.

\begin{proposition}[Observation-Calibrated $\phi$-Divergence DRO Reducible to CM-OT]
\label{prop:ConditionalDivergence}
For every instance of \eqref{eq:conditional_model}, there is a CM-OT instance \eqref{eq:primal} with equal value. 
Further, for any solution $(V, W) \sim \mu$ to the instance of $\eqref{eq:primal}$, $\nu \coloneqq \mathbbm{E}_\mu\left[W \; \mid \; V\right] \hat\nu$ solves $\eqref{eq:conditional_model}$.
\end{proposition}
The proof of \cref{prop:ConditionalDivergence} follows from a straightforward extension of the $\phi$-DRO reduction \cite[Appendix F.2.1]{chen2026perturbation}, which represents the decision variable $\nu$ as a reweighting of $\hat\nu$ through the CM-OT weight $W$ and uses conditional Jensen's inequality for the reverse direction.
The construction of the CM-OT instance sets $(\hat V, \hat W \equiv 1) \sim \hat\mu \coloneqq\hat\nu\otimes\delta_1$, objective $f(u)\coloneqq\ell(v)\cdot w$, transport cost $c(\hat v, \hat w, v, w) \coloneqq \phi(w) + \iota_{v=\hat v}$,
$\breve\Gamma_{\hat\mu} = \{\gamma\in\Gamma_{\hat\mu} : V=\hat V,\,\gamma\text{-a.s.,}\; \mathbbm{E}_\gamma[\abs{W}]<+\infty,\; \mathbbm{E}_\gamma[\phi(W)]<+\infty$,
and keeps $H$ and $h$ as given.

Using \cref{assn:**a,assn:**b,ass:local-conditional-ui}, we can design verifiable conditions for the observation-calibrated model \eqref{eq:conditional_model} to obtain a duality result through \cref{prop:main_result}.
Our presentation mirrors that of \cref{sec:Closing Duality Gap}.
{%

\begin{phiassumption}
\label{ass:1phi}
    There exists $q^0\in L^1_+(\hat\nu)$ such that
    \begin{equation*}
    \mathbbm E_{\hat\nu}[q^0(\hat V)\mid H(\hat V)]
    =
    h(H(\hat V)),
    \qquad
    \mathbbm E_{\hat\nu}[\phi(q^0)]<\rho,
    \qquad
    \mathbbm E_{\hat\nu}[q^0\ell_-]<+\infty.
    \end{equation*}
\end{phiassumption}

}%
{%
\begin{philemma}
\label{lem:1phi}
If \cref{ass:1phi} holds, then \cref{assn:**a} holds.
In particular, if
\[
\mathbbm E_{\hat\zeta}[\phi(h)]<\rho
\qquad\text{and}\qquad
\mathbbm E_{\hat\nu}[h(H(\hat V))\ell_-(\hat V)]<+\infty,
\]
then \cref{ass:1phi} holds.
\end{philemma}
}%
\begin{proof}
Let $\gamma^0$ be the law of
$(\hat V,1,\hat V,q^0(\hat V))$.
Then
\(
A\gamma^0 = \mathbbm E_{\hat\nu}[q^0(\hat V)\mid H(\hat V)] -h(H(\hat V)) =0,
\)
while
$\mathbbm E_{\gamma^0}[c]=\mathbbm E_{\hat\nu}[\phi(q^0)]<\rho$
and
$\mathbbm E_{\gamma^0}[f_-]=\mathbbm E_{\hat\nu}[q^0\ell_-]<+\infty$.
Hence $\gamma^0$ verifies \cref{assn:**a}.
The final claim can be seen by taking $q^0=h\circ H$.
\end{proof}

{%
\begin{phiassumption}
    \label{ass:2phi}
    It holds that $h(z)\in\operatorname{int}\operatorname{dom}\phi$, $\hat\zeta\text{-a.s.}$
\end{phiassumption}
}%

{%
\begin{philemma}
    \label{lem:2phi}
    If \cref{ass:1phi,ass:2phi} hold, then \cref{assn:**b} holds.
\end{philemma}
}%
\begin{proof}
By the pointwise interiority of $h$, we may choose a measurable $\beta:\mathcal Z\to(0,1]$ such that $\beta\leq h/2$ and $h\pm\beta\in\operatorname{dom}\phi$, $\hat\zeta$-a.s.
Set
\[
\alpha(v)\coloneqq
\frac{1}{1+\phi((h+\beta)(H(v)))+\phi((h-\beta)(H(v)))+(h+\beta)(H(v))\ell^-(v)}
\]
and define
\[
q^+
\coloneqq
q^0+\alpha\bigl[(h+\beta)\circ H-q^0\bigr]_+,
\qquad
q^-
\coloneqq
q^0-\alpha\bigl[q^0-(h-\beta)\circ H\bigr]_+.
\]
Each edited value is either $q^0$ or a convex combination of $q^0$ and $(h\pm\beta)\circ H$.
Hence $q^\pm\in L^1_+(\hat\nu)$, and convexity together with the definition of $\alpha$ gives
\(
\mathbbm E_{\hat\nu}[\phi(q^\pm)]<+\infty,
\)
and
\(
\mathbbm E_{\hat\nu}[q^\pm\ell^-]<+\infty.
\)
Moreover, $\mathbbm E[q^0\mid H]=h\circ H$ implies that $q^0<(h+\beta)\circ H$ and $q^0>(h-\beta)\circ H$ each occur with positive conditional probability, $\hat\zeta$-a.s.
Since $\alpha>0$,
\(
\mathbbm E_{\hat\nu}[q^+\mid H]>h\circ H
\)
and
\(
\mathbbm E_{\hat\nu}[q^-\mid H]<h\circ H,
\)
$\hat\zeta\text{-a.s.}$
The corresponding couplings $(\hat V,1,\hat V,q^\pm(\hat V))\sim\gamma^\pm$ therefore verify \cref{assn:**b}.    
\end{proof}

{%
\begin{phiassumption}
\label{assn:observation-local-tail}
There exists $b\in\mathbbm R$ such that, for every $a>0$, there is a finite set $A_a\subseteq\mathcal A_{\hat\zeta}$ satisfying
\begin{equation}
\label{eq:observation-local-tail}
\int_{\{v:H(v)\notin A_a\}}
\phi^*\!\left(a\,[\ell(v)-b]_+\right)\,\hat\nu(dv)
<+\infty,
\end{equation}
where $\mathcal A_{\hat\zeta} \coloneqq \{z\in\mathcal Z:\hat\zeta(\{z\})>0\}$.
\end{phiassumption}
}%

{%
\begin{philemma}[Verification of \cref{ass:local-conditional-ui}]
\label{lem:observation-local-a3}
Consider the CM-OT representation of the observation-calibrated $\phi$-divergence model.
Suppose that \cref{assn:**a,assn:**b} hold and that
$R_{\mathsf H}(\ell)<+\infty$.
If \cref{assn:observation-local-tail} holds, then \cref{ass:local-conditional-ui} holds with $\ell$ replaced by $\ell-b$.
\end{philemma}
}%
When $\hat\zeta$ has no atoms, Fenchel's inequality involving $\phi$ and $\phi^*$ gives a uniform bound that verifies both conditions of \cref{ass:local-conditional-ui}.
Additional details are deferred to Appendix~\ref{appendix:applications}; in particular, the case when $\hat\zeta$ has atoms uses a conditional mixing argument, similar to the one used in the proof of \cref{prop:main_result}. 
We then obtain the following.

\begin{corollary}[Observation-calibrated $\phi$-divergence duality]
\label{cor:observation-phi-duality}
Suppose that $(\V,\mathcal G)$ is standard Borel, $\mathcal Z$ is Polish, and Assumptions~\ref{ass:1phi}, \ref{ass:2phi}, and \ref{assn:observation-local-tail} hold.
Then
\begin{equation}
\label{eq:observation-phi-dual}
\begin{split}
R_{\mathsf H}(\ell)
={}&
\inf_{\substack{\lambda>0\\
\psi\in L^\infty(\hat\zeta)}}
\left\{
\lambda\rho+
\mathbbm E_{\hat\nu}\!\left[
\psi(H(\hat V))h(H(\hat V))
+
\lambda\phi^*\!\left(
\frac{\ell(\hat V)-\psi(H(\hat V))}{\lambda}
\right)
\right]
\right\}.
\end{split}
\end{equation}
Additionally, if $\textnormal{[(i)}\lor \textnormal{(ii)]}$ of \cref{thm:Cb_Linfty} holds, then $L^\infty(\hat\zeta)$ may be replaced with $\mathcal{C}_b(\Z)$.
\end{corollary}

\begin{proof}
By \cref{prop:ConditionalDivergence}, consider the corresponding CM-OT representation.
If $R_{\mathsf H}(\ell)=+\infty$, the result follows from weak duality, so suppose that $R_{\mathsf H}(\ell)<+\infty$.
Let $b$ be as in \cref{assn:observation-local-tail}, replace $\ell$ by $\ell-b$, and relabel the centered loss as $\ell$.
Since every feasible density has expectation one, this decreases both the primal and dual values by $b$.

By \cref{lem:observation-local-a3}, \cref{assn:observation-local-tail} verifies \cref{ass:local-conditional-ui}.
Thus \cref{prop:main_result} gives normality.
The interchange follows from \cref{prop:ip} .
For $\lambda>0$, the diagonal transport cost restricts the pointwise supremum to $v=\hat V$, and hence
\[
\sup_{w\geq0}
\left\{
w\bigl(\ell(\hat V)-\psi(H(\hat V))\bigr)-\lambda\phi(w)
\right\}
=
\lambda\phi^*\!\left(
\frac{\ell(\hat V)-\psi(H(\hat V))}{\lambda}
\right).
\]
Since the CM-OT transport cost is nonnegative, restricting the infimum to $\lambda>0$ does not change its value.
Shifting the loss and multiplier back by $b$ gives \eqref{eq:observation-phi-dual}.
\end{proof}
\noindent 

We observe that \cref{assn:observation-local-tail} admits simple interpretations for several standard divergences.
Letting $X\coloneqq[\ell-b]_+$, we summarize in \cref{tab:observation-local-tail} some common combinations of divergence and loss structure that verify \cref{assn:observation-local-tail}. These conditions concern only \cref{assn:observation-local-tail}; the remaining hypotheses, e.g., \cref{ass:1phi,ass:2phi}, of \cref{cor:observation-phi-duality} must be verified separately.

\begin{table}[ht]
\centering
\renewcommand{\arraystretch}{1.2}
\begin{tabularx}{\textwidth}{
@{}>{\raggedright\arraybackslash}p{0.25\textwidth}
>{\raggedright\arraybackslash}p{0.3\textwidth}
X@{}}
\toprule
\textbf{Structure}
&
\textbf{Ensuring \cref{assn:observation-local-tail}}
&
\textbf{Representative models}
\\
\midrule

$\hat\zeta$ finitely supported; any $\phi,\ell$
&
Take $A_a=\operatorname{supp}\hat\zeta$ for every $a>0$, so the integral in \eqref{eq:observation-local-tail} is zero.
&
Finite regimes, contexts, or strata, with possibly continuous and unbounded uncertainty within each stratum.
\\
\addlinespace

Any $\phi$; $\ell$ $\hat\nu$-essentially bounded above
&
Take $b\geq\operatorname*{ess\,sup}\ell$, so that $X=0$.
&
Failure or exceedance probabilities, including $\ell=\mathbbm 1_F$ with $b=1$; bounded classification losses; capped penalties.
\\
\addlinespace

Pearson $\chi^2$:
$\phi(w)=(w-1)^2$;
$X\in L^2(\hat\nu)$
&
For $s\geq0$, $\phi^*(s)=s+s^2/4$, so $\mathbbm E_{\hat\nu}[\phi^*(aX)]<+\infty$ for every $a>0$.
&
Inventory and resource-allocation models with square-integrable demand, forecast errors, or decisions.
\\
\addlinespace

Power:
$\phi(w)=w^p-pw+p-1$, $p>1$;
$X\in L^{p/(p-1)}(\hat\nu)$
&
Writing $q=p/(p-1)$, $\phi^*(s)$ has order $s^q$ on $\mathbbm R_+$, so the required conjugate moment is finite for every $a>0$.
&
Models with polynomially integrable losses, including heavier-tailed demand or forecast errors.
\\
\addlinespace

KL:
$\phi(w)=w\log w-w+1$;
$\mathbbm E_{\hat\nu}[e^{aX}]<+\infty$ $\forall a>0$
&
A convenient sufficient condition is: $\exists \beta>1$ s.t.~$\hat\nu(X>t)\leq Ce^{-ct^\beta}$.
&
Affine or Lipschitz losses under Gaussian uncertainty; losses with super-exponential upper tails; suitable count-data models such as Poisson losses.
\\
\addlinespace

Hard density cap:
$\phi=\iota_{[0,M]}$, $M>1$;
$X\in L^1(\hat\nu)$
&
For $s\geq0$, $\phi^*(s)=Ms$, so ordinary integrability of $X$ is sufficient.
&
Sensitivity models with bounded likelihood ratios and integrable, possibly unbounded, losses.
\\
\bottomrule
\end{tabularx}
\caption{Examples satisfying \Cref{assn:observation-local-tail}.}
\label{tab:observation-local-tail}
\end{table}

\paragraph{Specialization to $\phi$-divergence DRO}
We note that when $h \equiv 1$ and $H \equiv 0,$ then \eqref{eq:conditional_model} reduces to the familiar $\phi$-DRO problem \citep{bental2011robust}
\begin{align}
\label{eq:phi-dro}
\begin{array}{cl}
    \displaystyle \sup_{\nu\in\mathcal{P}(\V)} & \mathbbm{E}_{V\sim\nu}[\ell(V)] \\
    \st & \mathbbm{D}_\phi(\hat\nu,\nu) \leq \rho
\end{array}
\tag{$\rm {P}_{\phi}$}
\end{align}
with CM-OT dual

\begin{align*}
\inf_{\substack{\lambda>0\\
\psi\in \reals}}
\left\{
\lambda\rho+\psi +
\mathbbm E_{\hat\nu}\!\left[
\lambda\phi^*\!\left(
\frac{\ell(\hat V)-\psi}{\lambda}
\right)
\right]
\right\}
\label{eq:phi-dro-dual} \tag{$\text{D}_\phi+{\rm IP}$}.
\end{align*}
Indeed, \citet{bayraksan_love_2015} derive $\eqref{eq:phi-dro-dual}$ for the specific case of finite-supported $\hat\nu$. 

The primal-dual equality $\eqref{eq:phi-dro}=\eqref{eq:phi-dro-dual}$ has a well-documented result \citep{shapiro2017dro,KuhnShafieeWiesemann2025}; however, we note that our \cref{cor:observation-phi-duality} also yields the result immediately.

\begin{corollary}[$\phi$-Divergence DRO Duality \cite{shapiro2017dro,KuhnShafieeWiesemann2025}]
\label{cor:phi-dro}
Suppose that $(\V,\mathcal{G})$ is a standard Borel space.
If $\rho>0$, $1\in\intr\dom\phi$, and $\mathbbm E_{\hat\nu}[\ell(\hat V)]>-\infty$, then $\eqref{eq:phi-dro}=\eqref{eq:phi-dro-dual}$.
\end{corollary}
\begin{proof}[Proof of \cref{cor:phi-dro}]
Since $H \equiv 0,$ it follows that $\hat\zeta$ is finite-supported; hence, \cref{assn:observation-local-tail} holds. Since $h \equiv 1$, it then follows that \cref{ass:2phi} holds. Finally, we may define $q^0 \coloneqq h \circ H$ to see that \cref{ass:1phi} holds. Finally, note that $L^{\infty}(\hat\zeta)$ in this context is equivalent to $\reals$.
\end{proof}

\section{Conclusions}
This paper develops a repair-based approach to duality for CM-OT that, through the breadth of the CM-OT representation, applies to a range of distributionally robust optimization models. The central idea is to establish the closedness required for primal-dual equality by directly repairing nearly feasible couplings: suitable reference couplings correct vanishing constraint violations through mixing, while local conditional uniform integrability of the positive loss ensures that these corrections do not incur a nonvanishing deterioration in objective value. This perspective yields sufficient conditions that do not rely on compact uncertainty spaces or conventional interiority arguments and that permit unbounded losses under appropriate tail control. We further separate primal-dual equality from the additional analytical steps needed to obtain pointwise dual representations, identifying conditions under which essentially bounded measurable conditional-moment multipliers may be replaced, at the level of optimal values, by bounded continuous multipliers and under which optimization and integration may be interchanged. In this way, the analysis clarifies which regularity conditions are responsible for the distinct stages of CM-OT duality and provides a common route for studying duality across transport-, divergence-, and jointly perturbed DRO models with conditional restrictions. Finally, we apply the repair-based approach to a class of distributionally robust models that combines divergence-based ambiguity with distributional calibration, including applications in contextual resource allocation and reliability analysis.

\PaperAppendix

\section*{Developmental History}

Earlier versions under the arXiv identifier \href{https://arxiv.org/abs/2603.20503}{\texttt{2603.20503}}, most recently titled \emph{Perturbation Duality for Robust and Distributionally Robust Optimization: Short and General Proofs} by Louis L.~Chen, Jake Roth, and Johannes O.~Royset \cite{chen2026perturbation}, consisted of two distinct parts, one on DRO and one on RO.
The development of the DRO portion of the earlier work was led by Chen and Roth.
All three authors agreed that the two parts deserved their own dedicated treatment as standalone manuscripts, with authorship reflecting the contributions to each line of work.
The present manuscript continues the DRO portion.

\section*{AI Use Statement}
During the preparation of this manuscript and earlier versions, the authors used OpenAI's ChatGPT as an interactive research and editorial tool for literature-search assistance, review of proof strategies, and sharpening of results.
Throughout, the authors exercised their own judgment in incorporating suggestions from these models; responsibility for any errors remains with the authors.

\section{Interchangeability Proofs}

\begin{proof}[Proof of \cref{prop:ip}]
\label{pf:prop:ip}

Let ${\lambda>0}$ and $\psi\in\Psi \subseteq L^\infty(\hat\zeta)$. So that $J_{\lambda, \psi}$ is well-defined over the entirety of $\U\times\U$, we will without loss of generality assume that $\psi$ is bounded, measurable and $h$ is finite, measurable -- otherwise, we may  choose bounded Borel and finite Borel representatives for $\psi$ and $h$, respectively. By the discussions surrounding \eqref{eq:positive_lambda_identity} and \eqref{eq:ip-finish}, it suffices to show that $\mathcal{J}(\lambda, \psi) = \mathcal{J}^{IP}(\lambda, \psi),$ equivalently, $\sup_{\gamma \in \breve\Gamma_{\hat\mu}} \mathbbm{E}_{\gamma}[J_{\lambda, \psi}(\hat U, U)] = \mathbbm{E}_{\hat\mu}[\sup_{u \in \U} J_{\lambda, \psi}(\hat u, u)]$. As it is already clear that $\mathcal{J}(\lambda, \psi) \leq \mathcal{J}^{IP}(\lambda, \psi),$ we proceed to show the reverse, i.e., $\sup_{\gamma \in \breve\Gamma_{\hat\mu}}\mathbbm{E}_{\gamma}[J_{\lambda, \psi}(\hat U, U)] \geq \mathbbm{E}_{\hat\mu}[\sup_{u \in \U} J_{\lambda, \psi}(\hat u, u)].$ 
We will begin by approximating the pointwise supremum $\sup_{u \in \U} J_{\lambda, \psi}(\hat u, u)$ for all $\hat u$ outside a set of $\hat\mu$-measure zero via a measurable selection $\hat u \mapsto u$; so that we may eventually fashion a $\gamma \in \breve\Gamma_{\hat\mu}$ out of this selection, we will ensure that we always choose $(v,w) = u \in \U$ such that $\abs w$ and $c(\hat u; u)$ exhibit requisite boundedness.

\noindent \textbf{\underline{Selection and truncation}:}
{By the standard Borel and measurability hypotheses, $J_{\lambda,\psi}$ is extended-real Borel, measurable and hence upper semianalytic;}
{since $\lambda>0$ and $f$ never takes the value $+\infty$, $J_{\lambda,\psi}=-\infty$ wherever $c=+\infty$.}
For $m\in\mathbb N$, let
\[
    \U_m{(\hat u)}\coloneqq\{(v,w)\in\U:\abs w\leq m{,\;c(\hat u;v,w)\leq m}\},
    \qquad
    r_m(\hat u)\coloneqq\sup_{u\in\U_m{(\hat u)}}J_{\lambda,\psi}(\hat u,u)
\]
denote, respectively, an approximation of $\U$ and an approximation of the pointwise supremum value 
$r(\hat u)\coloneqq\sup_{u\in\U}J_{\lambda,\psi}(\hat u,u)$, for every $\hat u.$ {The set $\{(\hat u,u):\abs w\leq m,\ c(\hat u;u)\leq m\}$ is Borel, so $r_m$ is universally measurable.} The strategy, roughly speaking, is to select from $\U_m(\hat u)$ an approximate solution to $r_m(\hat u)$ for all $\hat u$; however, two corner cases complicate this -- more precisely, it is possible that there exist $\hat u$ such that $r_m(\hat u) \in \{+\infty, - \infty\}.$ Hence, we define $\bar r_m(\hat u)\coloneqq\min\{m,r_m(\hat u)\}$, and $S_m\coloneqq\{\hat u:r_m(\hat u)>-\infty\}$.
These truncations will afford arbitrarily accurate approximations via selection in the end, as $r_m\uparrow r$ and $\bar r_m\uparrow r$ pointwise. 

For every $m \in \mathbbm{N}$ and $\epsilon > 0$,
{by} the selection theorem of \cite[Proposition~7.50]{bertsekasshreve1978} (see also \cite[Theorem~2]{brownpurves1973}) applied to $\tilde J_{\lambda,\psi} \coloneqq \min\{m,J_{\lambda,\psi}\}$ over the graph of $\U_m$, there exists a universally measurable $x_{m,\epsilon}:\U\to{\U}$ such that for $\hat u \in S_m$ it holds that
\[
    {x_{m,\epsilon}(\hat u)\in\U_m(\hat u)}
    \qquad \text{ and } \qquad
    J_{\lambda,\psi}\bigl(\hat u,x_{m,\epsilon}(\hat u)\bigr)
    \geq  \tilde{J}_{\lambda,\psi}\bigl(\hat u,x_{m,\epsilon}(\hat u)\bigr) \geq \bar r_m(\hat u)-\epsilon,
\]
while for $\hat u \notin S_m,$ $x_{m,\epsilon}(\hat u)$ is set to an arbitrarily fixed member of $\U$. The selection $x_{m, \epsilon}$ yields the coupling
\[\mathscr{S}_{m, \epsilon} \coloneqq [\hat u \mapsto (\hat u, x_{m,\epsilon}(\hat u))]_{\#} \hat\mu.
\]

\noindent\textbf{\underline{Use the reference coupling}:}
We will now edit the couplings $\mathscr{S}_{m,\epsilon}$ using the reference coupling to construct $\gamma_{m,\epsilon}\in\breve\Gamma_{\hat\mu}$ whose objective values $\mathbbm E_{\gamma_{m,\epsilon}}[J_{\lambda,\psi}]$ approach $\mathbbm E_{\hat\mu}[r]$ as $m\to\infty$ and $\epsilon\downarrow0$.
We will make use of the reference $(\hat U, U) \sim \breve\gamma \in \breve\Gamma_{\hat\mu}$; in particular, since $\mathbbm E_{\breve\gamma}[\abs f]<\infty$, $h\in L^1(\hat\zeta)$, and $\psi$ is bounded, it follows that $J_{\lambda,\psi}\in L^1(\breve\gamma)$, so the conditional expectation $E_{\breve\gamma}\bigl[J_{\lambda,\psi}(\hat U,U)\mid \hat U\bigr]$ exists, and we'll let $b:\U\to\mathbb R$ be a finite $\mathcal F$-measurable version, i.e.,
\[
    b(\hat U)=\mathbbm E_{\breve\gamma}\bigl[J_{\lambda,\psi}(\hat U,U)\mid \hat U\bigr],
    \qquad \breve\gamma\text{-a.s.}
\]
In summary, $b$ is finite, $b\in L^1(\hat\mu)$, and since $J_{\lambda,\psi}(\hat U,U)\leq r(\hat U)$, we have $b\leq r$, $\hat\mu$-a.e. 

For $(m, \epsilon) \in \mathbbm{N} \times (0, +\infty),$ we define the subset $A_{m, \epsilon} \subseteq S_m$ via
\[
A_{m,\epsilon}\coloneqq\{\hat u\in S_m:J_{\lambda,\psi}(\hat u,x_{m,\epsilon}(\hat u))\geq b(\hat u)\}
\]
and the coupling $\gamma_{m, \epsilon} \in \breve\Gamma_{\hat\mu}$ via
\[
\gamma_{m,\epsilon}(B) \coloneqq \mathscr{S}_{m, \epsilon}\left(B \cap [A_{m, \epsilon} \times \U]\right) + \breve\gamma\left(B \cap [A^\complement_{m, \epsilon} \times \U]\right)\quad \forall B \in \mathcal{F} \otimes \mathcal{F}.
\]
In words, for $\hat u \notin A_{m,\epsilon}$,  $\gamma_{m,\epsilon}$ couples not with the selection $x_{m,\epsilon}(\hat u)$ but with those $u$ that $\breve\gamma$ would. To verify $\gamma_{m, \epsilon} \in \breve\Gamma_{\hat\mu},$ note that 
\begin{align*}
\gamma_{m,\epsilon}(E \times \U) &= \mathscr{S}_{m, \epsilon}\left([E \cap A_{m, \epsilon}] \times \U]\right) + \breve\gamma\left([E \cap A^\complement_{m, \epsilon}] \times \U\right)\\
&= \hat\mu\left(E \cap A_{m, \epsilon}\right) + \hat\mu\left(E \cap A_{m, \epsilon}^\complement\right) = \hat\mu\left(E\right) \quad\quad \forall E \in \mathcal{F}, 
\end{align*}
\[
\mathbbm E_{\gamma_{m,\epsilon}}[\abs W]\leq m+\mathbbm E_{\breve\gamma}[\abs W]<+\infty, \quad \text{ and } \quad \mathbbm E_{\gamma_{m,\epsilon}}[c_+]\leq m+\mathbbm E_{\breve\gamma}[c_+]<+\infty.
\]
Finally, $E_{\gamma_{m,\epsilon}}[J_{\lambda,\psi}]$ is well-defined, with $E_{\gamma_{m,\epsilon}}[[J_{\lambda,\psi}]_-] < +\infty;$ indeed, $[J_{\lambda,\psi}]_- \leq b_-$ over the event $[\hat U \in A_{m,\epsilon}]$, while 
\[
E_{\gamma_{m,\epsilon}}[[J_{\lambda,\psi}(\hat U,U)]_-; [\hat U \in A_{m,\epsilon}^\complement]] = E_{\breve\gamma}[[J_{\lambda,\psi}(\hat U,U)]_-; [\hat U \in A_{m,\epsilon}^\complement]]  < +\infty,
\]
since $J_{\lambda,\psi}\in L^1(\breve\gamma)$.

\noindent\textbf{\underline{Compare objective values}:}
Observe that by construction of $\gamma_{m, \epsilon},$
\begin{align*}
    \mathbbm E_{\gamma_{m,\epsilon}}[J_{\lambda,\psi}(\hat U,U)]
    &=
    \mathbbm E_{\hat\mu}\left[J_{\lambda,\psi}(\hat U,x_{m,\epsilon}(\hat U));\, [\hat U \in A_{m,\epsilon}]\right]
    +
    \mathbbm E_{\hat\mu}\left[b(\hat U);\, [\hat U \in A_{m,\epsilon}^\complement]\right]\\
    &\geq \mathbbm{E}_{\hat\mu}\left[\max(\bar{r}_m, b) - \epsilon \right].
\end{align*}
Indeed,  $J_{\lambda,\psi}(\hat u,x_{m,\epsilon}(\hat u)) \geq \max(b(\hat u) , \bar{r}_m(\hat u) - \epsilon)$ for $\hat u \in A_{m,\epsilon}.$ As well, $b(\hat u) > J_{\lambda,\psi}(\hat u,x_{m,\epsilon}(\hat u)) \geq \bar{r}_m(\hat u) - \epsilon$ when $\hat u \in A^\complement_{m,\epsilon} \cap S_m,$ and $b(\hat u) > -\infty = \bar{r}_m(\hat u)$ when $\hat u \in A^\complement_{m,\epsilon} \cap S_m^\complement$.

Since $\bar r_m\uparrow r$ and $b\leq r$, the nonnegative functions $\max\{\bar r_m,b\}-b$ increase {almost everywhere} to $r-b$.
Hence monotone convergence yields $\mathbbm E_{\hat\mu}[\max\{\bar r_m,b\}]\to\mathbbm E_{\hat\mu}[r]$.
{Letting first $\epsilon\downarrow0$ and then $m\to\infty$ proves the desired lower bound.}
\end{proof}

\begin{proof}[Proof of \cref{thm:Cb_Linfty}]
\label{pf:thm:Cb_Linfty}
Since $\mathcal C_b(\Z)\subseteq L^\infty(\hat\zeta)$, $(\dualLinfty)\leq(\dualCb)$;
thus, if $(\dualLinfty)=+\infty$, then $(\dualLinfty)=(\dualCb)=+\infty$.
Further, if $\breve\Gamma_{\hat\mu}=\varnothing$, then $(\dualLinfty)=(\dualCb)=-\infty$.
Thus, we proceed under the assumption that $\breve\Gamma_{\hat\mu}\neq\varnothing$ and $(\dualLinfty)<+\infty$.

For every $\gamma\in\breve\Gamma_{\hat\mu}$, the cost and moment penalties are integrable.
Thus, if $\mathcal J(\lambda,\psi)=-\infty$ for some pair, then $\mathbbm E_\gamma[f]=-\infty$ for every $\gamma\in\breve\Gamma_{\hat\mu}$, and consequently $\mathcal J(0,0)=-\infty$.
In this case $(\dualCb)=(\dualLinfty)=-\infty$.
We may therefore restrict attention to pairs $(\lambda,\psi)$ with $\mathcal J(\lambda,\psi)\in\R$.

Fix a pair $(\lambda,\psi)\in\reals_+\times L^\infty(\hat\zeta)$ with $\mathcal J(\lambda,\psi)\in\R$, set $K\coloneqq\norm{\psi}_\infty$, and without loss of generality assume $\abs\psi\leq K$.
By Lusin's and Tietze Extension theorems, for each $\delta>0$ there exists a closed $C_\delta\subseteq\Z$ and $\psi'_\delta\in\mathcal C_b(\Z)$ such that $\hat\zeta(C_\delta^\complement)<\delta$, $\psi'_\delta=\psi$ on $C_\delta$, and $\norm{\psi'_\delta}_\infty\leq K$.
In particular, $\abs{\psi-\psi'_\delta}\leq 2K\ind_{C_\delta^\complement}$.
We proceed to show that $\mathcal J(\lambda+\Delta,\psi'_\delta)$ exceeds $\mathcal J(\lambda,\psi)$ by a vanishingly small amount as $(\Delta,\delta)\downarrow0$; this will prove $(\dualLinfty)=(\dualCb)$.
If $K=0$, this is immediate, so we proceed under the assumption that $K>0$.

Under \textnormal{(i)}, for any $\Delta>0$,
\begin{align*}
&\mathcal J(\lambda+\Delta,\psi'_\delta)-\mathcal J(\lambda,\psi)
\leq \Delta\rho+\sup_{\gamma\in\breve\Gamma_{\hat\mu}}
\mathbbm E_\gamma\left[
-\Delta c+(\psi-\psi'_\delta)W-(\psi-\psi'_\delta)h(\hat Z)
\right]\\
&\leq \Delta(\rho-\inf c)
+2K\sup_{\gamma\in\breve\Gamma_{\hat\mu}}
\mathbbm E_\gamma\left[\abs W;\hat Z\in C_\delta^\complement\right]
+2K\mathbbm E_{\hat\zeta}\left[
\abs{h(\hat Z)};\hat Z\in C_\delta^\complement
\right]\\
&=\Delta(\rho-\inf c)
+2K\sup_{\gamma\in\breve\Gamma_{\hat\mu}}
\mathbbm E_{\hat\zeta}\left[
\mathbbm E_\gamma[\abs W\mid\hat Z];
\hat Z\in C_\delta^\complement
\right]
+2K\mathbbm E_{\hat\zeta}\left[
\abs{h(\hat Z)};\hat Z\in C_\delta^\complement
\right],
\end{align*}
which tends to zero as both $\Delta$ and $\delta$ tend to zero by uniform integrability and $h\in L^1(\hat\zeta)$.

Under \textnormal{(ii)}, for any $\Delta>0$,
\begin{align*}
&\mathcal J(\lambda+\Delta,\psi'_\delta)-\mathcal J(\lambda,\psi)
\leq \Delta\rho+\sup_{\gamma\in\breve\Gamma_{\hat\mu}}
\mathbbm E_\gamma\left[
-\Delta c+(\psi-\psi'_\delta)W-(\psi-\psi'_\delta)h(\hat Z)
\right]\\
&\leq \Delta\rho+\sup_{\gamma\in\breve\Gamma_{\hat\mu}}
\left\{
\mathbbm E_\gamma[-\Delta c]
+2K\mathbbm E_\gamma\left[
a_{\frac{\Delta}{2K}}(\hat U)
+\frac{\Delta}{2K}c;
\hat Z\in C_\delta^\complement
\right]
\right\}
+2K\mathbbm E_{\hat\zeta}\left[
\abs{h(\hat Z)};\hat Z\in C_\delta^\complement
\right]\\
&\leq \Delta\left(
\rho+\mathbbm E_{\hat\mu}[a_1(\hat U)]
\right)
+2K\mathbbm E_{\hat\mu}\left[
a_{\frac{\Delta}{2K}}(\hat U);
\hat Z\in C_\delta^\complement
\right]
+2K\mathbbm E_{\hat\zeta}\left[
\abs{h(\hat Z)};\hat Z\in C_\delta^\complement
\right],
\end{align*}
where the final inequality follows from $0\leq a_1(\hat u)+c(\hat u;u)$ in \textnormal{(ii)}.
For fixed $\Delta>0$, the last two terms vanish as $\delta\downarrow0$ by integrability, and then letting $\Delta\downarrow0$ gives the desired approximation.

Thus, under either \textnormal{(i)} or \textnormal{(ii)}, every $\mathcal J(\lambda,\psi)$ with finite value can be approximated arbitrarily closely in dual value by a pair $(\lambda',\psi')\in\reals_+\times\mathcal C_b(\Z)$.
Taking infima over $\lambda,\psi$ (along a minimizing sequence if $(\dualLinfty)=-\infty$) then gives $(\dualCb)\leq(\dualLinfty)$.
Taken together with the reverse inequality above, $(\dualCb)=(\dualLinfty)$.
\end{proof}

\section{Supplemental Examples on \texorpdfstring{\cref{assn:**a,assn:**b,ass:local-conditional-ui}}{Assumptions \ref*{assn:**a}, \ref*{assn:**b}, and \ref*{ass:local-conditional-ui}}}
Here we collect examples clarifying the role and relation to other regularity conditions of \cref{assn:**a,assn:**b,ass:local-conditional-ui}.
\cref{apx:A1A2} focuses on the separate roles of the assumptions. In particular, \cref{ex:duality_gap_a-not-b} shows that \cref{assn:**b} is important: a duality gap may exist when \cref{assn:**a,ass:local-conditional-ui} hold yet \cref{assn:**b} fails, and then shows how \cref{assn:**b} eliminates the gap in that example.
Second, \cref{apx:nonexistence} compares \cref{assn:**a,assn:**b,ass:local-conditional-ui} with the traditional Slater's condition and any other condition that guarantees solution existence.
\cref{ex:not_slater} illustrates that our assumptions are distinct from Slater's condition, while
\cref{ex:nonexistence_primal,lem:dual-nonexist-template,ex:nonexistence_dual,ex:nonexistence_dual_Cb} illustrate that they are distinct from conditions that guarantee solution existence (primal or dual).

\subsection{Role of \texorpdfstring{\cref{assn:**b}}{Assumption \ref*{assn:**b}}}
\label{apx:A1A2}
\begin{example}[Duality gap when \cref{assn:**b} fails]
    \label{ex:duality_gap_a-not-b}
    Let
    $\V\coloneqq\{0\}$ and
    $\W\coloneqq[0,1]$
    be endowed with standard topologies, $\mathcal F$ be the product of their respective Borel $\sigma$-algebras,
    $f(v,w)\coloneqq \ind_{\{w<1\}}(w)$,
    $h\coloneqq1$,
    $\rho\coloneqq 1 + \epsilon$ for any $\epsilon>0$,
    and $c(\hat u,u)\coloneqq (w-\hat w)^2$.
    Also let $\hat\nu\coloneqq\delta_0$, $\hat\mu\coloneqq \hat\nu\otimes\delta_{0}$, and $\Psi\coloneqq L^\infty(\hat\nu)=\R$.

    \cref{assn:**a} holds by considering the coupling $\gamma^0\in\Gamma_{\hat\mu}$ with $W\equiv1$, as $\mathbbm{E}_{\gamma^0}[W-1\mid\hat V]=0$, $\hat\nu$-a.s.~and $\mathbbm{E}_{\gamma^0}[c] = \mathbbm{E}[W^2]=1<\rho$.
    However, \cref{assn:**b} fails since for any coupling $\gamma\in\Gamma_{\hat\mu}$, $\W$ requires $0\leq W\leq 1$, $\gamma$-a.s.,~so $\mathbbm{E}_\gamma[W-1\mid\hat V]\leq0$ $\hat\nu$-a.s.
    Thus it is impossible to find $\gamma^+\in\Gamma_{\hat\mu}$ with $\mathbbm{E}_{\gamma^+}[W-1\mid\hat V]\geq b>0$.

    In this setting, $\eqref{eq:primal}=\eqref{eq:primal_explicit}$ by \cref{lem:v1=v2},
    and because any feasible coupling must have $W\equiv1$, we find that $\eqref{eq:primal} = 0$, yet \begin{align*}
        \dualPsiRef
        =
        \dualPsiIPRef
        =
        \inf_{\lambda\geq0,\psi\in\R}
        \lambda\rho + \sup_{w\in[0,1]} \ind_{\{w<1\}}(w) - \lambda w^2 - \psi\cdot(w-1)=1,
    \end{align*}
    since along any sequence $1\neq w\uparrow1$, for any $\lambda\geq0$ and $\psi\in\R$, it holds that $\ind_{\{w<1\}}(w)=1$, $\lambda (\rho-w^2)\to\lambda\epsilon\geq0$, and $\psi\cdot(w-1)\to0$.
\end{example}
While \cref{ex:duality_gap_a-not-b} uses the fact that $f$ is not upper semicontinuous, the failure of \cref{assn:**b} is also important.
One way to satisfy \cref{assn:**b} and close the gap would be to perturb the upper bound of $\W$ by a small amount.
Indeed, redefining $\W\coloneqq[0,1+\sqrt{\epsilon}]$ for $\epsilon\in(0,1]$,
it would then be possible to perturb the (formerly unique) feasible $W\equiv1$ ``up'' and ``down'' within the revised domain to $ W^+ \coloneqq W + \sqrt{\epsilon}$ and $W^- \coloneqq W - \sqrt{\epsilon}$.
Moreover, for a sequence $\delta_r\downarrow0$ with  $\delta_r\in(0,\sqrt{\epsilon}]$, we may then define a sequence of couplings $\{\gamma^r\}_{r\geq1}$ with $W_r$-marginal obtained by mixing $W^+$ and $W_r^-\coloneqq W - \delta_r$ independently of $W$ drawn according to $t_r\sim{\rm Bern}(q_r)$ where $q_r\coloneqq \delta_r/(\delta_r + \sqrt{\epsilon})\in(0,1)$, i.e., $W_r\coloneqq (1-t_r) W^-_r + t_r W^+$.
Then $\mathbbm{E}_{\gamma^r}[f]\to1$ as $r\to\infty$ and both $\mathbbm{E}_{\gamma^r}[W_r - 1\mid\hat V]=0$ and $\mathbbm{E}_{\gamma^r}[c]=1+\sqrt{\epsilon}\delta_r \leq 1+\epsilon=\rho$ for all $r$ when $\delta_r\in(0,\sqrt{\epsilon}]$, yielding $\eqref{eq:primal}=1$ to match $\dualPsiRef$.
This example motivates the form of \cref{assn:**b}.

\subsection{Relation of \texorpdfstring{\cref{assn:**a,assn:**b}}{Assumptions \ref*{assn:**a} and \ref*{assn:**b}} to Conditions Guaranteeing Solution Existence}
\label{apx:nonexistence}
First we provide an example demonstrating that \cref{assn:**a,assn:**b} are distinct from the traditional Slater-type conditions.
By construction, our examples will satisfy \cref{ass:local-conditional-ui} as well.
\begin{example}[\cref{assn:**a,assn:**b} are Distinct from Slater]
    \label{ex:not_slater}
    Let
    $\V\coloneqq\reals$ and
    $\W\coloneqq\R_+$ be endowed with standard topologies, $\mathcal F$ be the product of their respective Borel $\sigma$-algebras,
    $f\equiv 0$ ,
    $h\equiv1$,
    $\rho>0$,
    $c\equiv0$,
    $(\hat V, \hat W) \sim \hat \mu$ be such that the marginal $\hat V \sim \hat \nu \coloneqq N(0,1).$
    Consider the perturbation function $p:L^1(\hat\nu)\to\Reals$
    \begin{align*}
        p(\theta)=
        \left\{\!
        \begin{array}{cl}
            \sup_{\gamma\in\breve\Gamma_{\hat\mu}}
            & \mathbbm{E}_\gamma[f] \\
            \st & \mathbbm{E}_\gamma[W\mid\hat V]= 1 + \theta(\hat V), \;\; \hat V \text{-a.e.}
        \end{array}
        \!\right\}.
    \end{align*}
    Define a collection of functions $\{\theta_r\}_{r=1}^\infty \subseteq L^1(\hat \nu)$. For each positive integer $r,$ let $\theta_r(\hat v)$ take value $-2$ if $\hat v\in[r,r+1]$ and zero otherwise.
    Then $\|\theta_r\|_{L^1(\hat \nu)} = 2\cdot \hat \nu ([r, r+1]) \downarrow 0$ as $r \rightarrow \infty,$ and yet for every integer $r,$ it is the case that for all
    $\gamma \in \breve\Gamma_{\hat\mu}$,
    \[
    \mathbbm{E}_\gamma[W\mid\hat V] - 1 \geq 0 -1 > -2 = \theta_r(\hat V)
    \]
    with measure $\hat \nu([r,r+1]) > 0$.
    In other words, the Slater condition \cref{prop:convex_duality}\ref{cond:slater} does not hold, i.e., $0 \notin \intr \dom p$.
    With some mild additional assumptions on $f$, we further obtain $0 \notin \ri \dom p$.
    Indeed, $\dom p \subseteq \{\theta\in L^1(\hat\nu) : \theta \geq -1,\;\text{$\hat\nu$-a.s.}\} = -1+L_+^1(\hat\nu)$, so when $\dom p \supseteq -1+L_+^1(\hat\nu)$, then  $\aff \dom p = \aff L^1_+(\hat\nu) = L^1(\hat\nu)$, i.e., $\ri \dom p = \intr_{\aff \dom p} \dom p = \intr_{L^1(\hat\nu)}\dom p = \intr \dom p$.
    For example, given the coupling $\gamma_\theta$ with $W=1+\theta(\hat V)$, any condition on $f$ ensuring that $\mathbbm{E}_{\gamma_\theta}[f]>-\infty$ for all $\theta\in -1+L_+^1(\hat\nu)$ is sufficient for $\dom p \supseteq -1+L_+^1(\hat\nu)$.
    However, \cref{assn:**a,assn:**b} hold. 
\end{example}
\noindent It can be shown that \cref{ex:not_slater} satisfies a more generalized interiority condition from \cite{bot2012regularity}, which guarantees dual attainment, but this may not be true in general.

We now take up the issue of primal and dual attainment.
To start, we present an example showing that the existence of primal optimal solutions is not guaranteed when the compactness of $\U$ is relaxed and \cref{assn:**a,assn:**b} hold, even if $f$ and $c$ remain closed, i.e., $[\neg\ref{assn:U_compact}\land\ref{assn:f_closed}\land\ref{assn:c_closed}]$ from \cref{sec: Conjecture}.
\begin{example}[Primal nonexistence]
    \label{ex:nonexistence_primal}
    Let
    $\V=\{0\}$ and 
    $\W\coloneqq\R$ be endowed with the standard (subspace) topologies,
    $\mathcal{F}$ be the product of their respective Borel $\sigma$-algebras,
    $f(v,w) \coloneqq -1/(1+\abs{w})$,
    $h\coloneqq1$,
    $\rho>0$,
    $c \coloneqq 0$,
    $\hat \mu \coloneqq \delta_{(0,0)}$,
    and $\Psi\coloneqq L^\infty(\hat\nu)=\mathcal{C}_b(\V)=\R$.
    Then $0=\eqref{eq:primal}=\dualPsiRef=\dualPsiIPRef$, but no optimal $\gamma$ exists.

    Clearly \cref{assn:**a,assn:**b} hold, so $\eqref{eq:primal}=\eqref{eq:primal_explicit}=\dualPsiRef$ by \cref{prop:main_result}.
    We find $\dualPsiRef=\dualPsiIPRef$ by noting that $f,c,\psi$ are continuous and applying
    \cref{prop:ip}.
    Finally it is clear that $f(0,w)<0$ for all $w\in\R$, so $\mathbbm{E}_\mu[f]<0$ for any coupling $\mu$ that is feasible to \eqref{eq:primal}, preventing primal attainment.
\end{example}

Next we examine the possibility that the dual problem $\dualPsiRef$ fails to admit an optimal solution.
Towards this, we recall that:
\(
C_b(\V) \subseteq L^\infty(\hat\nu) \subseteq L^1(\hat\nu).
\)
Specifically, we will demonstrate that it is possible for the optimal value $(\dualLinfty)$ to be only approached in the limit by members of $L^\infty(\hat\nu)$ whilst being attained at $\lambda=0$ by a unique member of $L^1(\hat\nu)$. Similarly, the optimal value $(\dualCbIP)$ may only approached in the limit by members of $C_b(\V)$ whilst being  attained at $\lambda=0$ by a unique member of $L^\infty(\hat\nu)$.

We also note that nonattainment in $\dualPsiRef$ implies nonattainment in $\dualPsiIPRef$ whenever $\dualPsiRef=\dualPsiIPRef=d\in\R$.
Denoting the objective of $\dualPsiIPRef$ by $\mathcal J^{\ip}$, we have $\mathcal J \leq \mathcal J^{\ip}$, so any $(\lambda^*,\psi^*)$ attaining $d$ in $\dualPsiIPRef$ satisfies
\[
d \leq \mathcal J(\lambda^*,\psi^*) \leq \mathcal J^\ip(\lambda^*,\psi^*)=d,
\]
so $\argmin \mathcal J^\ip \subseteq \argmin \mathcal J$.
The inclusion may be strict, since a minimizer $(0,\psi^*)$ of $\mathcal J$ may satisfy $\mathcal J(0,\psi^*)=d<\mathcal J^\ip(0,\psi^*)$ and thus fail to minimize $\mathcal J^\ip$.

The following result will provide a useful template for constructing instances in which no dual optimal solution to {\dualPsiRef} exists and yet $\eqref{eq:primal}=\dualPsiRef=\dualPsiIPRef$.

\begin{lemma}[Template for dual nonexistence]
    \label{lem:dual-nonexist-template}
    Let $\V$ be a Polish space with Borel $\sigma$-algebra $\mathcal{G}$ and probability measure $\hat\nu$, and let $\W\coloneqq[0,2]$ carry its usual subspace topology.
    Let $\mathcal F$ be the product of their respective Borel $\sigma$-algebras.
    Let $\hat\mu \coloneqq \hat\nu \otimes \delta_0 \in \mathcal{P}(\U)$.
    Let $h\equiv1$, $H$ be the identity, $\rho>0$,
    and $c = \iota_{v=\hat v}$.
    Finally, let $f(v,w)\coloneqq g(v)(1-w/2)$, where $g:\V\to(-\infty,0]$ is $\mathcal{G}$-measurable with $g\in L^1(\hat\nu)$.

    Then $\eqref{eq:primal} = (\dualLinfty) = (\dualLinftyIP)$, and in addition, $(\lambda^*,\psi^*) = (0,-g/2)$ is the $\hat\nu$-a.e.~unique solution to $(\dualLone)$.
\end{lemma}

\begin{proof}[Proof of \cref{lem:dual-nonexist-template}]
    In the following, let $\Psi = L^\infty(\hat\nu)$.
    By construction of $c\in\{0,+\infty\}$, \[\breve\Gamma_{\hat\mu} = \{\gamma\in\Gamma_{\hat\mu} : V=\hat V,\;\gamma\text{-a.s.}\}.\tag{$\dagger$}\]
    We now verify \cref{assn:**a,assn:**b}; since $f\leq0$, \cref{ass:local-conditional-ui} holds automatically.
    Let $\gamma^0$ be the coupling defined by $\hat V\sim\hat\nu$, $\hat W\equiv 0$, $V=\hat V$ a.s., and $\gamma^0(W=0\mid\hat V)=\gamma^0(W=2\mid\hat V)=1/2$, $\hat\nu$-a.s., so $\mathbbm{E}_{\gamma^0}[W-1\mid\hat V]=0$, $\hat\nu$-a.s.
    Then $\mathbbm{E}_{\gamma^0}[c]=0<\rho$ and $\mathbbm{E}_{\gamma^0}[f]=\mathbbm{E}_{\hat\nu}[g]/2>-\infty$ since $g\in L^1(\hat\nu)$.
    Thus \cref{assn:**a} holds, so $\eqref{eq:primal} = \eqref{eq:primal_explicit}$ reduces to a problem over couplings.
    Let $\gamma^\pm$ be couplings with $\hat V\sim\hat\nu$, $\hat W\equiv 0$, $V=\hat V$ a.s., and $W^\pm\equiv1\pm1$.
    Then $\mathbbm{E}_{\gamma^\pm}[c]=0$, $\mathbbm{E}_{\gamma^\pm}[\abs{f}]\leq\mathbbm{E}_{\hat\nu}[\abs{g}]<+\infty$, and $\mathbbm{E}_{\gamma^\pm}[W^\pm-1\mid\hat V]\equiv\pm1$, so \cref{assn:**b} holds.

    If $\gamma$ is feasible for \eqref{eq:primal_explicit}, then the conditional moment constraint $\mathbbm{E}_\gamma[W\mid\hat V]=1$ implies $\mathbbm{E}_\gamma[1-W/2\mid\hat V]=1/2$, $\hat\nu$-a.s.
    Since $\abs{f(V,W)}\leq\abs{g(\hat V)}$ and $g\in L^1(\hat\nu)$, every feasible $\gamma$ has the same objective value, so $\mathbbm{E}_{\gamma}[g(\hat V)(1-W/2)]=\mathbbm{E}_{\hat\nu}[g]/2=\eqref{eq:primal}=\dualPsiRef$, where the last equality follows from \cref{prop:main_result}.

    By $(\dagger)$, $c=0$ $\gamma$-a.s.\ for every $\gamma\in\breve\Gamma_{\hat\mu}$, so for all $\lambda\geq0$ and $\psi\in\Psi$, \[\mathcal J(\lambda,\psi)=\lambda\rho+B(\psi),\qquad B(\psi)\coloneqq\sup_{\gamma\in\breve\Gamma_{\hat\mu}}\mathbbm E_\gamma\bigl[f-\psi(\hat V)(W-1)\bigr].\]
    Since $\rho>0$, it follows that \[\dualPsiRef=\inf_{\lambda\geq0,\psi\in\Psi}\mathcal J(\lambda,\psi)=\inf_{\psi\in\Psi}B(\psi)=\inf_{\lambda>0,\psi\in\Psi}\mathcal J(\lambda,\psi).\]

    Next we show that the IP holds in terms of value equality.
    To begin, for any $\gamma\in\breve\Gamma_{\hat\mu}$, let $t:\V\to[0,1]$ be a $\mathcal G$-measurable version of $\mathbbm{E}_\gamma[W\mid\hat V=\cdot]/2$.
    Since $f(\hat v,w)-\psi(\hat v)(w-1)$ is affine in $w$, its expectation under $\gamma$ agrees with that under the coupling Since $f(\hat v,w)-\psi(\hat v)(w-1)$ is affine in $w$, its expectation is the same under $\gamma$ and $\gamma_t$, where
    \[
    \gamma_t(d\hat v,d\hat w;dv,dw)
    =
    \hat\nu(d\hat v)\,\delta_0(d\hat w)\,\delta_{\hat v}(dv)
    \bigl((1-t(\hat v))\,\delta_0(dw)+t(\hat v)\,\delta_2(dw)\bigr).
    \]
    Conversely, every $\mathcal G$-measurable $t:\V\to[0,1]$ defines such a coupling $\gamma_t\in\breve\Gamma_{\hat\mu}$.
    Thus the inner supremum over couplings reduces to a supremum over $\mathcal G$-measurable $t$, and for fixed $\psi\in\Psi$, satisfies
    \begin{align*} &\sup_{\gamma\in\breve\Gamma_{\hat\mu}}\mathbbm{E}_\gamma[f(V,W)-\psi(\hat V)\!\cdot\!(W-1)] \\ &=\!\!\!\sup_{t:\V\to[0,1]}\medint\bigl((1-t(\hat v))\!\cdot\![f(\hat v,0)+\psi(\hat v)]+t(\hat v)\!\cdot\![f(\hat v,2)-\psi(\hat v)]\bigr)\,\hat\nu(d\hat v) \\ &=\medint\max\{g(\hat v)+\psi(\hat v),-\psi(\hat v)\}\,\hat\nu(d\hat v).
    \end{align*}
    The last equality uses the pointwise identity $\sup_{t\in[0,1]}(1-t)a+tb=\max\{a,b\}$, with the maximizing choice $t^*(\hat v)=\ind_{\{-\psi(\hat v)\geq g(\hat v)+\psi(\hat v)\}}(\hat v)$, which is $\mathcal G$-measurable since it is defined by comparing $\mathcal G$-measurable functions.
    For every $\lambda>0$, the diagonal cost restricts $v=\hat v$, and the affine dependence on $w\in[0,2]$ makes the pointwise supremum equal to the maximum of the endpoint values, so $\mathcal J^\ip(\lambda,\psi)=\lambda\rho+B(\psi)$.
    Together with $\mathcal J\leq\mathcal J^\ip$ and $\inf_{\lambda>0}\lambda\rho=0$, this justifies the IP, and \begin{align*} \dualPsiRef=\dualPsiIPRef&=\inf_{\psi\in\Psi}\mathbbm{E}_{\hat\nu}[\max\{g+\psi,-\psi\}] \\ &=\mathbbm{E}_{\hat\nu}[g]/2+\inf_{\psi\in\Psi}\mathbbm{E}_{\hat\nu}\left[\abs{\psi+g/2}\right]\geq\mathbbm{E}_{\hat\nu}[g]/2. \end{align*}
    The last equality uses the pointwise identity $\max\{g(\hat v)+\psi(\hat v),-\psi(\hat v)\}=g(\hat v)/2+\abs{\psi(\hat v)+g(\hat v)/2}$.
    The same calculation of $B(\psi)$ holds for every $\psi\in L^1(\hat\nu)$, since $g\in L^1(\hat\nu)$, $\abs{f(V,W)}\leq\abs{g(\hat V)}$, and $\abs{W-1}\leq1$, $\gamma$-a.s.\ for every $\gamma\in\breve\Gamma_{\hat\mu}$.
    Thus $\psi^*=-g/2$ is the $\hat\nu$-a.e.~unique minimizer of $B$ over $L^1(\hat\nu)$.
    Since $\rho>0$, the unique optimal solution for $(\dualLone)$ is then $(\lambda^*,\psi^*)=(0,-g/2)$.
\end{proof}

 Using \cref{lem:dual-nonexist-template}, we construct  an example in which solutions in $\Psi\coloneqq L^\infty(\hat\nu)$ can approximate the dual optimal value $(\dualLinfty)$ but no individual member can attain it. 
\begin{example}[Dual nonexistence in $L^\infty(\hat\nu)$]
    \label{ex:nonexistence_dual}
    In the setting of \cref{lem:dual-nonexist-template},
    let
    $\V\coloneqq(0,1]\subset\R$,
    $\hat\nu\coloneqq{\rm Unif}(0,1]$,
    and $g\in L^1(\hat\nu)\setminus L^\infty(\hat\nu)$ be given by $g(v)\coloneqq-1/\sqrt{v}$ so $f\leq0$.
    Then $\eqref{eq:primal}=(\dualLinfty)=(\dualLinftyIP)=\mathbbm{E}_{\hat\nu}[g/2]$ but no pair $(\lambda,\psi) \in \R_+\times L^\infty(\hat\nu)$ solves $(\dualLinfty)$.
    
    By \cref{lem:dual-nonexist-template}, the $\hat\nu$-a.e.~unique minimizer of $B$ over $L^1(\hat\nu)$ is $\psi^*=-g/2 \in L^1(\hat\nu)\setminus L^\infty(\hat\nu)$. However, the optimal dual value $\mathbbm{E}_{\hat\nu}[g/2]$ can be approached using a sequence $(0,\psi_r)$ with $\{\psi_r\}_{r=1}^\infty\subseteq L^\infty(\hat\nu)$ in which $\psi_r\coloneqq-g_r/2$, where $g_r\coloneqq\max\{g,-r\}$. Indeed, $g_r\to g$ in $L^1(\hat\nu)$.
\end{example}

\cref{lem:dual-nonexist-template} can also be used to produce an
instance in which there is no dual optimal solution in the setting of \citet[Theorem 4.2]{jiajin2025unifying} with $\Psi\coloneqq\mathcal{C}_b(\V)$, even when $\U$ is compact, $f$ is upper semicontinuous, and $c$ lower semicontinuous, i.e., $[\ref{assn:U_compact}\land\ref{assn:f_closed}\land\ref{assn:c_closed}]$ from \cref{sec: Conjecture} holds.

\begin{example}[Dual nonexistence in $\mathcal{C}_b(\V)$]
\label{ex:nonexistence_dual_Cb}
    In the setting of \cref{lem:dual-nonexist-template}, let $\V\coloneqq[0,1]$, $\hat\nu\coloneqq{\rm Unif}[0,1]$, and $\Psi\coloneqq L^\infty(\hat\nu)$.
    Let $C\subseteq\V$ be the (closed) Smith-Volterra-Cantor set, for which $\hat\nu(C)=1/2$, and $O\coloneqq\V\setminus C$ denote its complement, which is open in $\R$.
    Finally, set $g(v)\coloneqq-\ind_O(v)\in L^\infty(\hat\nu)\setminus\mathcal{C}_b(\V)$.
    Then $\eqref{eq:primal}=\dualPsiRef=\dualPsiIPRef=\mathbbm{E}_{\hat\nu}[g/2]$.

    Also by \cref{lem:dual-nonexist-template}, the $\hat\nu$-a.e.~unique minimizer of $B$ is $\psi^*=-g/2$, so $(0,\psi^*)$ attains {\dualPsiRef}.
    If we restrict the dual multiplier space to $\mathcal{C}_b(\V)$, then the optimal value is preserved, but there does not exist $\psi'\in\mathcal{C}_b(\V)$ achieving the optimal value.
    Indeed, for ${\rm dist}(\cdot,C)$ denoting the distance to $C$, define $g_r(v)\coloneqq-\min\{1,r\cdot{\rm dist}(v,C)\}\in\mathcal{C}_b(\V)$ (since ${\rm dist}(\cdot,C)$ is continuous) and $\psi_r\coloneqq-g_r/2$, so that $g_r\downarrow g$ and $\psi_r\uparrow\psi^*$, with $\mathcal J(0,\psi_r)$ approaching $\mathbbm{E}_{\hat\nu}[g/2]$ from above by the monotone convergence theorem.
    However, if $(\lambda,\psi')\in\R_+\times\mathcal{C}_b(\V)$ attained the optimal value, then $\psi'=\psi^*=\tfrac12\ind_O$, $\hat\nu$-a.e.
    Since $O$ is open and $\hat\nu$ has full support, continuity implies that $\psi'=1/2$ on $O$; otherwise an open subset of $O$ would have positive $\hat\nu$-mass with $\psi'\neq1/2$.
    Since $O$ is dense, continuity gives $\psi'\equiv1/2$ on $[0,1]$, contradicting that $\psi'=0$ $\hat\nu$-a.e.~on $C$ where $\hat\nu(C)=1/2$, so $\psi^*$ is not attained in $\mathcal{C}_b(\V)$.

    Moreover, $f(v,w)=-\ind_O(v)(1-w/2)$ is upper semicontinuous on $\V\times[0,2]$.
    Indeed, $\ind_O$ is nonnegative and lower semicontinuous since $O$ is open, while $w\mapsto1-w/2$ is nonnegative and continuous on $[0,2]$.
    Hence $-f$ is lower semicontinuous.
    This verifies closedness of $f$ \ref{assn:f_closed}; moreover, compactness of $\U$ \ref{assn:U_compact} and closedness of $c$ \ref{assn:c_closed} both hold.
\end{example}

\section{Supplement for \texorpdfstring{\cref{sec:: Applications}}{Section \ref*{sec:: Applications}}}
\label{appendix:applications}
\begin{proof}[Proof of \cref{lem:observation-local-a3}]
Consider first when $\hat\zeta$ has no atoms.
After centering the loss by $b$, Fenchel's inequality gives, for every $a>0$,
\(
W\ell_+
\leq
\tfrac{1}{a}\phi(W)
+
\tfrac{1}{a}\phi^*(a\ell_+).
\)
Taking $a=1$ and using $\mathbbm E_\gamma[\phi(W)]\leq\rho+1$ for $\gamma\in\mathcal N_1$, together with \cref{assn:observation-local-tail}, gives the uniform $L^1(\hat\zeta)$ bound in \cref{ass:local-conditional-ui}.
For the small-set condition, let $E\subseteq\mathcal Z$.
The same inequality gives
\[
\sup_{\gamma\in\mathcal N_1}
\mathbbm E_\gamma[\ell_+\cdot W;\hat Z\in E]
\leq
\frac{\rho+1}{a}
+
\frac{1}{a}
\mathbbm E_{\hat\nu}
\!\left[
\phi^*(a\,\ell_+(\hat V));H(\hat V)\in E
\right].
\]
For fixed $a$, the second term tends to zero as $\hat\zeta(E)\downarrow0$ by integrability, and then letting $a\to+\infty$ makes the first term vanish.
This proves the small-set condition in \cref{ass:local-conditional-ui}.

Now suppose that $\hat\zeta$ has atoms.
The same Fenchel argument with $a=1$ controls the positive loss outside the finite set $A_1$, so it remains only to obtain a uniform bound on each atom $z\in A_1$.
Fix such a $z$ and any $\gamma\in\mathcal N_1$.
We will make two edits to $\gamma$ to construct a feasible coupling $\tilde\gamma$ that retains a fixed positive fraction of
\(
\mathbbm E_\gamma[W\ell_+(\hat V);\hat Z=z]
\)
while keeping its negative objective contribution uniformly bounded.
Since $\tilde\gamma$ is feasible, its objective is at most $R_{\mathsf H}(\ell)$, which will give the desired uniform bound at $z$.

\noindent\textbf{\underline{First edit}:}
Define $\bar\gamma$ by the conditional mixture that uses $\gamma$ on
$\{\hat Z=z,\ell(\hat V)\geq0\}$ and the strictly feasible coupling $\gamma^0$ from \cref{assn:**a} elsewhere.
Then
\[
\mathbbm E_{\bar\gamma}[W\ell_+(\hat V);\hat Z=z]
=
\mathbbm E_\gamma[W\ell_+(\hat V);\hat Z=z],
\]
and $(A\bar\gamma)(z)$ may not be zero,
but $A\bar\gamma=0$ away from $z$.
Since $c\geq0$, $\gamma\in\mathcal N_1$, and $\gamma^0$ has finite cost,
\[
\mathbbm E_{\bar\gamma}[c]
\leq
\mathbbm E_\gamma[c]+\mathbbm E_{\gamma^0}[c]
\leq
\rho+1+\mathbbm E_{\gamma^0}[c].
\]
Moreover, $\bar\gamma$ uses $\gamma$ only where $\ell(\hat V)\geq0$, so its negative-loss contribution comes only from $\gamma^0$ and
\[
\mathbbm E_{\bar\gamma}[f_-]
\leq
\mathbbm E_{\gamma^0}[f_-]
<+\infty.
\]
Thus $\bar\gamma$ has uniformly bounded transport cost and negative loss.
Thus the first edit preserves the positive-loss contribution we wish to bound while leaving only the scalar residual $A\bar\gamma(z)$ to be repaired.
The uniform bounds on the cost, negative loss, and $A\bar\gamma(z)$ ensure that this remaining repair can be performed with mixing weights, budget cost, and negative objective contribution controlled uniformly over $\gamma\in\mathcal N_1$.

\noindent\textbf{\underline{Second edit}:}
For either sign, define $\gamma_z^\pm$ by the conditional mixture that uses $\gamma^\pm$ from \cref{assn:**b} on $\{\hat Z=z\}$ and $\gamma^0$ elsewhere.
Then
\[
A\gamma_z^\pm=0
\quad\text{away from }z,
\qquad
A\gamma_z^+(z)=A\gamma^+(z)>0,
\qquad
A\gamma_z^-(z)=A\gamma^-(z)<0.
\]
Hence, if $A\bar\gamma(z)>0$, we use $\gamma_z^-$, while if $A\bar\gamma(z)<0$, we use $\gamma_z^+$.
Starting with a fixed fraction $s_z>0$ of $\bar\gamma$, mix in exactly enough of the appropriate $\gamma_z^\pm$ to cancel the residual $s_zA\bar\gamma(z)$, and assign the remaining mass to $\gamma^0$.
The resulting coupling $\tilde\gamma$ therefore satisfies $A\tilde\gamma=0$ everywhere.

Because $A\bar\gamma(z)$ is uniformly bounded and $A\gamma_z^\pm(z)$ are fixed and strictly separated from zero, the amount of $\gamma_z^\pm$ required is bounded by a fixed multiple of $s_z$.
We may therefore choose $s_z>0$, independently of $\gamma$, sufficiently small that all mixture weights are nonnegative and the strict budget slack of $\gamma^0$ ensures that $\tilde\gamma$ satisfies the divergence budget.
Hence $\tilde\gamma$ is feasible.

\noindent\textbf{\underline{Conclusion}:}
Finally, $\tilde\gamma$ retains the fraction $s_z$ of the original positive-loss contribution at $z$, while its negative-loss contribution is bounded by a constant $K_z<+\infty$ independent of $\gamma$.
Therefore
\[
R_{\mathsf H}(\ell)
\geq
\mathbbm E_{\tilde\gamma}[W\ell(\hat V)]
\geq
s_z\mathbbm E_\gamma[W\ell_+(\hat V);\hat Z=z]-K_z.
\]
Since $R_{\mathsf H}(\ell)<+\infty$, this gives
\(
\sup_{\gamma\in\mathcal N_1}
\mathbbm E_\gamma[W\ell_+(\hat V);\hat Z=z]
<+\infty.
\)
Together with the bound outside $A_1$, these estimates establish
\[
\sup_{\gamma\in N_1}
\mathbb{E}_{\gamma}\bigl[W\ell^+(\hat V)\bigr]
<+\infty.
\]
It remains to verify the small-set condition. 
Let $\epsilon>0$ and note that
every $E\in\mathcal{H}$ with sufficiently small
$\hat\zeta(E)$ must satisfy $E\cap A_a=\varnothing$.
For such an $E$, Fenchel's inequality, nonnegativity of $\phi$,
and the bound
$\mathbb{E}_{\gamma}[\phi(W)]\leq\rho+1$ for $\gamma\in N_1$
give
\[
\sup_{\gamma\in N_1}
\mathbb{E}_{\gamma}
    \bigl[W\ell^+(\hat V);\,\hat Z\in E\bigr]
\leq
\frac{\rho+1}{a}
+
\frac{1}{a}
\int_{H(\hat v) \in E}
    \phi^*\bigl(a\ell^+(\hat v)\bigr)\,\hat\nu(d\hat v).
\]
Choosing $a$ sufficiently large and then $E$ with sufficiently small $\hat\zeta(E)$ will make this upper bound smaller than $\epsilon.$ 
\end{proof}

    \PaperEndAppendix

      \bibliographystyle{\PaperBibStyle}
\bibliography{refs-standardized}%

\end{document}